\documentclass{amsart}
\usepackage{amssymb}
\usepackage[mathscr]{eucal}
\usepackage{hyperref}
\usepackage{url}
\usepackage[all]{xy}
\usepackage{color}

\makeatletter
\let\@wraptoccontribs\wraptoccontribs
\makeatother

\newtheorem{thm}[equation]{Theorem}
\newtheorem{lem}[equation]{Lemma}
\newtheorem{cor}[equation]{Corollary}
\newtheorem{prop}[equation]{Proposition}

\newtheorem{question}[equation]{Question}
\newtheorem*{coruncount}{Corollary \ref{uncount}}
\newtheorem*{corshlap}{Corollary \ref{shlap}}
\newtheorem*{thmg0}{Theorem \ref{genus0}}

\theoremstyle{definition}
\newtheorem{rem}[equation]{Remark}
\newtheorem{rems}[equation]{Remarks}
\newtheorem{defn}[equation]{Definition}
\newtheorem{exa}[equation]{Example}
\newtheorem*{notation}{Notation}

\numberwithin{equation}{section}

\newfont{\cyrr}{wncyr10}
\def\Sh{\mbox{\cyrr Sh}}

\def\bA{\mathbb{A}}
\def\Z{\mathbb{Z}}
\def\Q{\mathbb{Q}}
\def\F{\mathbb{F}}
\def\G{\mathbb{G}}
\def\pp{\mathbb{P}}
\def\R{\mathbb{R}}
\def\C{\mathbb{C}}

\def\bN{\mathbf{N}}

\def\ba{\mathbf{a}}

\def\Qp{\Q_p}
\def\Ql{\Q_ \ell}

\def\Fl{\F_\ell}

\def\A{\mathcal{A}}
\def\O{\mathcal{O}}
\def\P{\mathcal{P}}

\def\H{\mathcal{H}}
\def\I{\mathcal{I}}

\def\cF{\mathcal{F}}
\def\L{\mathcal{L}}

\def\cS{\mathcal{S}}
\def\T{\mathcal{T}}

\def\cQ{\mathcal{Q}}
\def\cC{\mathcal{C}}

\def\ld{\mathcal{h}}
\def\rd{\mathcal{i}}

\def\a{\mathfrak{a}}

\def\p{\mathfrak{p}}
\def\b{\mathfrak{b}}

\def\d{\mathfrak{d}}

\def\Hom{\mathrm{Hom}}
\def\Gal{\mathrm{Gal}}
\def\rk{\mathrm{rank}}
\def\lth{\mathrm{length}}

\def\inv{\mathrm{inv}}

\def\Res{\mathrm{Res}}
\def\End{\mathrm{End}}
\def\Spec{\mathrm{Spec}}
\def\Frob{\mathrm{Fr}}
\def\image{\mathrm{image}}
\def\Sel{\mathrm{Sel}}

\def\Aut{\mathrm{Aut}}
\def\loc{\mathrm{loc}}
\def\cond{\mathrm{cond}}
\def\tors{\mathrm{tors}}
\def\ur{\mathrm{ur}}

\def\PSL{\mathrm{PSL}}
\def\GL{\mathrm{GL}}
\def\SL{\mathrm{SL}}
\def\ab{\mathrm{ab}}
\def\Br{\mathrm{Br}}

\def\M{\mathcal{M}}
\def\OO{\mathcal{R}}
\def\EE{\mathcal{E}}
\def\Kb{\bar{K}}

\def\EElambda{\M_\lambda\EE}

\def\too{\longrightarrow}
\def\map#1{\;\xrightarrow{#1}\;}
\def\isom{\xrightarrow{\sim}}

\def\hookto{\hookrightarrow}
\def\onto{\twoheadrightarrow}
\def\dirsum#1{\underset{#1}{\textstyle\bigoplus}}

\def\bmu{\boldsymbol{\mu}}
\def\HS#1{H^1_{#1}}
\def\Hu{\HS{\ur}}

\def\OK{\O_K}

\def\SH{\mathrm{S}H}
\def\TT{T}

\title{Diophantine stability}

\author{Barry Mazur}
\author{Karl Rubin}

\contrib[with an appendix by]{Michael Larsen}


\thanks{This material is based upon work supported by the 
National Science Foundation under grants DMS-1302409, DMS-1065904, and DMS-1101424.
Mazur and Larsen would also like to thank MSRI for support and hospitality.}

\begin{document}

\begin{abstract}
If $V$ is an irreducible algebraic variety over a number field $K$,
and $L$ is a field containing $K$, we say that 
$V$ is {\em diophantine-stable} for $L/K$ if $V(L) = V(K)$.  
We prove that if $V$ is either a simple abelian variety, or a 
curve of genus at least one, then under mild hypotheses there is a set $S$
of rational primes with positive density such that for every $\ell \in S$ 
and every $n \ge 1$, there are infinitely many cyclic 
extensions $L/K$ of degree $\ell^n$ for which $V$ is diophantine-stable.
We use this result to study the collection of finite extensions of $K$ 
generated by points in $V(\Kb)$.
\end{abstract}

\maketitle

\tableofcontents

\part{Introduction, conjectures and results}
\label{part1}
\section{Introduction}
\label{intro}
Throughout Part \ref{part1} (\S\ref{intro} through \S\ref{uncpf}) we fix a number field $K$.

\subsection{Diophantine stability}

For any field $K$, we denote by $\bar{K}$ a fixed separable closure of $K$, and by 
$G_K$ the absolute Galois group $\Gal(\Kb/K)$.

\begin{defn}
Suppose $V$ is an irreducible algebraic variety over $K$.
If $L$ is a field containing $K$, we say that 
$V$ is {\em diophantine-stable} for $L/K$ if $V(L) = V(K)$.  

If $\ell$ is a rational prime, 
we say that $V$ is {\em $\ell$-diophantine-stable} over $K$ if for every 
positive integer $n$, and every finite set $\Sigma$ of places of $K$, 
there are infinitely many cyclic extensions $L/K$ of degree $\ell^n$,
completely split at all places $v \in \Sigma$, such that $V(L) = V(K)$.
\end{defn}

The main results of this paper are the following two theorems.

\begin{thm}
\label{static-avs} 
Suppose $A$ is a simple abelian variety over $K$ and all $\Kb$-endo\-morphisms of  
$A$ are defined over $K$.  
Then there is a set $S$ of rational primes with positive density such that $A$ is 
$\ell$-diophantine-stable over $K$ for every $\ell \in S$. 
\end{thm}

\begin{thm}
\label{static-curves} 
Suppose  $X$ is an irreducible curve over $K$, and let $\tilde{X}$ be the normalization 
and completion of $X$.  If $\tilde{X}$ has genus $\ge 1$, 
and all $\Kb$-endomorphisms of the jacobian of $\tilde{X}$ 
are defined over $K$, then 
there is a set $S$ of rational primes with positive density such that $X$ is 
$\ell$-diophantine-stable over $K$ for every $\ell \in S$. 
\end{thm}

\begin{rems}
(1) Note that our assumptions on $A$ imply that $A$ is absolutely simple.
It is natural to ask whether the assumption on $\End(A)$ is necessary, and whether 
the assumption that $A$ is simple is necessary.  See Remark \ref{nonsimple} 
for more about the latter question.

(2) The condition on the endomorphism algebra in Theorems \ref{static-avs} and 
\ref{static-curves} can always be satisfied by enlarging $K$.

(3) For each $\ell \in S$ in Theorem \ref{static-avs} and each $n \ge 1$, 
Theorem \ref{quantthm} below 
gives a quantitative lower bound for the number of cyclic extensions 
of degree $\ell^n$ and bounded conductor for which $A$ is $\ell$-diophantine-stable.
\end{rems}

We will deduce Theorem \ref{static-curves} from Theorem \ref{static-avs} 
in \S\ref{proof1} below, and prove the following consequences in \S\ref{uncpf}.  
Corollary \ref{uncount} is proved by applying Theorem \ref{static-curves} repeatedly 
to the modular curve $X_0(p)$, 
and Corollary \ref{shlap} by applying Theorem \ref{static-curves} repeatedly 
to an elliptic curve over $\Q$ of positive rank and using results of Shlapentokh.

\begin{cor}  
\label{uncount}
Let $p\ge 23$ and $p\ne 37, 43, 67, 163$. There are uncountably many 
pairwise non-isomorphic subfields $L$ of $\bar{\Q}$ such that no elliptic 
curve defined over $L$ possesses an $L$-rational subgroup 
of order $p$.
\end{cor}

\begin{cor}
\label{shlap}
For every prime $p$, there are uncountably many pairwise non-isomorphic 
totally real fields $L$ of 
algebraic numbers in $\Qp$ over which the following two statements both hold:
\begin{enumerate} 
\item 
There is a diophantine definition of ${\Z}$ in the ring of integers $\O_L$ of $L$. 
In particular, Hilbert's Tenth Problem has a negative answer for $\O_L$; 
i.e., there does not exist an algorithm to determine whether a polynomial 
(in many variables) with coefficients in $\O_L$ has a solution in $\O_L$.
 \item 
There exists a first-order definition of the ring $\Z$ in $L$. 
The first-order theory for such fields $L$ is undecidable.
\end{enumerate}
\end{cor}

\subsection{Fields generated by points on varieties}
Our original motivation for Theorem \ref{static-curves} was to understand, 
given a variety $V$ over $K$, the set of (necessarily finite) extensions of 
$K$ generated by a single $\Kb$-point of $V$.  More precisely, we make the following 
definition.

\begin{defn}   
Suppose $V$ is a variety defined over $K$.  
A finite extension $L/K$ is {\em generated over $K$ by a point of $V$} 
if (any of) the following equivalent conditions hold:
\begin{itemize} 
\item 
There is a point $x\in V(L)$ such that $x \notin V(L')$ for any proper subextension $L'/K$. 
\item  
There is an $x \in V(\Kb)$ such that $L=K(x)$.
\item 
There is an open subvariety $W \subset V$, an embedding $W \hookto {\mathbb A}^N$ 
defined over $K$, and a point in the image of $W$ whose coordinates generate $L$ over $K$.
\end{itemize}
If $V$ is a variety over $K$ we will say that {\em $L/K$ belongs to $V$} 
if $L/K$ is generated by a point of $V$ over $K$. 
Denote by $\L(V;K)$ the set of finite extensions of $K$ belonging to $V$, that is:
$$
\L(V;K) := \{K(x)/K : x \in V({\Kb})\}.
$$
\end{defn}

For example, if $V$ contains a curve isomorphic over $K$ to 
an open subset of $\pp^1$, then it follows from the primitive element theorem 
that every finite extension of 
$K$ belongs to $V$.  It seems natural to us to conjecture the converse.  We 
prove this conjecture for irreducible curves.  Specifically:

\begin{thm}
\label{curves}
Let $X$ be an irreducible curve over $K$. Then the following are equivalent:
\begin{enumerate}
\item
all but finitely many finite
extensions $L/K$ belong to $X$, 
\item
$X$ is birationally isomorphic (over $K$) to the projective line.
\end{enumerate}
\end{thm}

Theorem \ref{curves} is a special case of Theorem \ref{genus0} below, taking $Y = \pp^1$.

More generally, one can ask to what extent $\L(X;K)$ determines the curve 
$X$. 

\begin{question}
\label{ques} 
Let $X$ and $Y$ be irreducible smooth projective 
curves over  a number field $K$. 
If ${\L}(X;K) = {\L}(Y;K)$, are $X$ and $Y$ necessarily isomorphic over ${\Kb}$?
\end{question}
 
With $\Kb$ replaced by $K$ in Question \ref{ques}, the answer is ``no''.  
A family of counterexamples  
found by Daniel Goldstein and Zev Klagsbrun is given in Proposition \ref{GK} below. 
However, Theorem \ref{genus0} below shows that a stronger version of  
Question \ref{ques} has a positive answer if $X$ has genus zero.

We will write $\L(X;K) \approx \L(Y;K)$ to mean that $\L(X;K)$ and $\L(Y;K)$ 
agree up to a finite number of elements, i.e., the symmetric difference
$$
{\L}(X;K) \cup {\L}(Y;K) - {\L}(X;K) \cap {\L}(Y;K)
$$ 
is finite.  

We can also ask Question \ref{ques} with ``$=$'' replaced by ``$\approx$''. 
Lemma \ref{binv} below shows that up to ``$\approx$'' equivalence, $\L(X;K)$ is 
a birational invariant of the curve $X$.

\begin{thm}
\label{genus0}
Suppose $X$ and $Y$ are irreducible curves over $K$, and 
$Y$ has genus zero.  Then $\L(X;K) \approx \L(Y;K)$ if and only if $X$ and $Y$ 
are birationally isomorphic over $K$.
\end{thm}

Theorem \ref{genus0} will be proved in \S\ref{curvesect}.

\subsection{Growth of Mordell-Weil ranks in cyclic extensions}
Fix an abelian variety $A$ over $K$.
Theorem \ref{static-avs} produces a large number of cyclic extensions $L/K$ 
such that $\rk(A(L)) = \rk(A(K))$.  
For fixed $m \ge 2$, it is natural to ask how ``large'' is the set 
$$
\cS_m(A/K) := \{\text{$L/K$ cyclic of degree $m$} : \rk(A(L)) > \rk (A(K))\}.
$$
In \S\ref{quant} we use the proof of Theorem \ref{static-avs} to 
give quantitative information about the size of $\cS_{\ell^n}(A/K)$ for prime powers $\ell^n$. 

Conditional on the Birch and Swinnerton-Dyer Conjecture, 
$\cS_m(A/K)$ is closely related to the collection of $1$-dimensional characters $\chi$ of $K$ 
of order dividing $m$ such that the $L$-function $L(A,\chi; s)$ of the 
abelian variety $A$ twisted by $\chi$ has a zero at the central point $s= 1$.  
There is a good deal of literature on the statistics of such zeroes, particularly 
in the case where $A=E$ is an elliptic curve over ${\Q}$. 
For $\ell$ prime let
$$
N_{E, \ell}(x):= |\{\text{Dirichlet characters $\chi$ of order $\ell$ : $\cond(\chi) \le x$ and $L(E,\chi,1) = 0$}\}|.
$$
David, Fearnley and Kisilevsky \cite{DFK} conjecture that
$
\lim_{x \to \infty}N_{E,\ell}(x) 
$
is infinite for $\ell \le 5$, and finite for $\ell \ge 7$.  More precisely, 
the Birch and Swinnerton-Dyer Conjecture would imply
$$
\log N_{E,2}(x) \sim \log(x),\\
$$
and David, Fearnley and Kisilevsky \cite{DFK} conjecture that as $x \to \infty$, 
$$
\log N_{E, 3}(x) \sim {\textstyle\frac{1}{2}}\log(x),\quad
\log N_{E, 5}(x) \ll_{\epsilon} \epsilon\log(x) ~\text{for all $\epsilon > 0$}.
$$

Examples with $L(E,\chi,1) = 0$ for $\chi$ of large order $\ell$ seem to be 
quite rare over ${\Q}$.  Fearnley and Kisilevsky  \cite{FK} provide examples 
when ${\ell}=7$ and one example with $\ell=11$ (the curve 
$E : y^2 + xy = x^3 + x^2 -32x + 58$ of conductor $5906$, with $\chi$ of conductor $23$). 
  
In contrast, working over more general number fields there can be a large 
supply of cyclic extensions $L/K$ in which the rank grows.  
We will say that a cyclic extension $L/K$ is of {\em dihedral type} if there are 
subfields $k \subset K_0 \subset K$ and $L_0 \subset L$ such that $[K_0:k] = 2$, 
$L_0/k$ is Galois with dihedral Galois group, and $KL_0 = L$.
The rank frequently grows in extensions of dihedral type, 
as can be detected for parity reasons, and sometimes buttressed by Heegner 
point constructions.  See \cite[\S2, \S3]{growth} and \cite[Theorem B]{alc}.
This raises the following natural question.

\begin{question}
\label{1.12}
Suppose $V$ is either an abelian variety or an irreducible curve of genus 
at least one over $K$.  
Is there a bound $M(V)$ such that if $L/K$ is cyclic of degree $\ell > M(V)$ and not 
of dihedral type, then $V(L) = V(K)$?
\end{question}

A positive answer to Question \ref{1.12} for abelian varieties implies a positive answer 
for irreducible curves of positive genus, exactly as Theorem \ref{static-curves} 
follows from Theorem \ref{static-avs} (see \S\ref{proof1}).

\subsection{Outline of the paper}
In \S\ref{curvesect} we prove Theorem \ref{genus0}.
The rest of Part \ref{part1} is devoted to deducing Theorem \ref{static-curves} 
from Theorem \ref{static-avs}, and deducing Corollary \ref{uncount} 
from Theorem \ref{static-curves}.
The heart of the paper is Part \ref{part2} (sections \ref{twists} through \ref{hyps}), 
where we prove Theorem \ref{static-avs}.  In \S\ref{quant} we give  
quantitative information about the number of extensions $L/K$  
relative to which our given abelian variety is diophantine-stable. 

Here is a brief description of the strategy of the proof of Theorem \ref{static-avs} 
in the case when $\End(A) = \Z$ and $n = 1$.  (For a more thorough description 
see \S\ref{intro2}, the introduction to Part \ref{part2}.)  
The strategy in the general case is similar, but must deal with the complexities 
of the endomorphism ring of $A$.
If $L/K$ is a cyclic extension of degree $\ell$, we show (Proposition \ref{tower}) that 
$\rk(A(L)) = \rk(A(K))$ if and only if a certain Selmer group 
we call $\Sel(L/K,A[\ell])$ vanishes.  The Selmer group $\Sel(L/K,A[\ell])$ is a subgroup 
of $H^1(K,A[\ell])$ cut out by local conditions $\H_\ell(L_v/K_v) \subset H^1(K_v,A[\ell])$ 
for every place $v$, that depend on the local extension $L_v/K_v$.  Thus finding $L$ 
with $A(L) = A(K)$ is almost the same as finding $L$ with ``good local conditions'' 
so that $\Sel(L/K,A[\ell]) = 0$.

If $v$ is a prime of $K$, not dividing $\ell$, where $A$ has good reduction, 
we call $v$ ``critical'' if $\dim_{\Fl}A[\ell]/(\Frob_v-1)A[\ell] = 1$, and ``silent'' 
if $\dim_{\Fl}A[\ell]/(\Frob_v-1)A[\ell] = 0$.
If $v$ is a critical prime, then the local condition $\H_\ell(L_v/K_v)$ 
only depends on whether $L/K$ is ramified at $v$ or not.  If $v$ is a silent prime, 
then $\H_\ell(L_v/K_v) = 0$ and does not depend on $L$ at all.
Given a sufficiently large supply of critical primes, we show (Propositions \ref{goodp} 
and \ref{l7.13}) how to choose a finite set $\Sigma_c$ of critical primes so that 
if $\Sigma_s$ is any finite set of silent primes, $L/K$ is completely split at all 
primes of bad reduction and all primes above $\ell$, and the set of primes ramifying in 
$L/K$ is $\Sigma_c \cup \Sigma_s$, then $\Sel(L/K,A[\ell]) = 0$.

The existence of critical primes and silent primes for a set of rational primes $\ell$ 
with positive density is Theorem \ref{main} of the Appendix by Michael Larsen.
We are very grateful to Larsen for providing the Appendix, and to 
Robert Guralnick, with whom we consulted and who patiently explained much of the theory to us.  
We also thank Daniel Goldstein and Zev Klagsbrun for Proposition \ref{GK} below.

\section{Fields generated by points on varieties}
\label{curvesect}

Recall that for a variety $V$ over $K$ we have defined
$$
\L(V;K) := \{K(x)/K : x \in V({\Kb})\}.
$$

\subsection{Brauer-Severi varieties}

Suppose that $X$ is a variety defined over $K$ and isomorphic 
over $\Kb$ to $\pp^n$, i.e., $X$ is an $n$-dimensional Brauer-Severi variety.  
Let $\Br(K) := H^2(G_K,\Kb^\times)$ denote the Brauer group of $K$.  
As a twist of $\pp^n$, $X$ corresponds to a class in $H^1(G_K,\Aut_{\Kb}(\pp^n))$, 
so using the map 
$$
H^1(G_K,\Aut_{\Kb}(\pp^n)) = H^1(G_K,\PSL_{n+1}(\Kb)) \hookto H^2(G_K,\bmu_{n+1}) = \Br(K)[n+1]
$$
$X$ determines (and is determined up to $K$-isomorphism by) a class 
$$
c_X \in \Br(K)[n+1].
$$
For every place $v$ of $K$, let $\inv_v : \Br(K) \to \Br(K_v) \to \Q/\Z$ 
denote the local invariant.

\begin{prop}
\label{b-s}
Suppose that $X$ is a Brauer-Severi variety over $K$, and let $c_X \in \Br(K)$ 
be the corresponding Brauer class.  If $L$ is a finite extension of $K$ then 
the following are equivalent:
\begin{enumerate}
\item
$X(L)$ is nonempty,
\item
$L \in \L(X;K)$,
\item
$[L_w:K_v]\inv_v(c_X) = 0$ for every $v$ of $K$ and every $w$ of $L$ above $v$.
\end{enumerate}
\end{prop}

\begin{proof}
Let $n := \dim(X)$, and suppose $X(L)$ is nonempty.
Then $X$ is isomorphic over $L$ to $\pp^n$.  If $K \subset F \subset L$ 
then the Weil restriction of scalars $\Res^F_K X$ is a variety of dimension $n[F:K]$, 
and there is a natural embedding
$$
\Res^F_K X \too \Res^L_K X.
$$
If we define $W := \Res^L_K X - \cup_{K \subset F \subsetneq L} \Res^F_K X$
then $W$ is a (nonempty) Zariski open subvariety of the rational variety $\Res^L_K X$, 
so in particular $W(K)$ is nonempty.  But taking $K$ points in the definition of 
$W$ shows that 
$$
W(K) = (\Res^L_K X)(K) - \cup_{K \subset F \subsetneq L} (\Res^F_K X)(K)
   = X(L) - \cup_{K \subset F \subsetneq L} X(F).
$$
Thus $X(L)$ properly contains $\cup_{K \subset F \subsetneq L} X(F)$, 
so $L \in \L(X;K)$ and (i) $\Rightarrow$ (ii).

If $v$ is a place of $K$ and $w$ is a place of $L$ above $v$, then 
(see for example \cite[Proposition 2, \S1.3]{cfs})
\begin{equation}
\label{e22}
\inv_w(\Res_L(c_X)) = [L_w:K_v]\inv_v(c_X).
\end{equation}
If $L \in \L(X;K)$, then by definition $X(L)$ is nonempty, so
$X$ is isomorphic over $L$ to $\pp^n$ and $\Res_L(c_X) = 0$.  
Thus \eqref{e22} shows that (ii) $\Rightarrow$ (iii).

Finally, if (iii) holds then $\inv_w(\Res_L(c_X)) = 0$ for 
every $w$ of $L$ by \eqref{e22}, so $\Res_L(c_X) = 0$ (see for example 
\cite[Corollary 9.8]{cft}).
Hence $X$ is isomorphic over $L$ to $\pp^n$, so $X(L)$ is nonempty
and we have (iii) $\Rightarrow$ (i).
\end{proof}

\begin{cor}
\label{bscor}
If $X$ and $Y$ are Brauer-Severi varieties, then $\L(X;K) = \L(Y;K)$ 
if and only if $\inv_v(c_X)$ and $\inv_v(c_Y)$ have the same denominator 
for every $v$.
\end{cor}

\begin{proof}
This follows directly from the equivalence (ii) $\Leftrightarrow$ (iii) 
of Proposition \ref{b-s}.
\end{proof}

\subsection{Curves}  For this subsection $X$ will be a curve over $K$, and we 
will prove Theorem \ref{genus0}.

\begin{lem}
\label{binv}
Suppose $X$ and $Y$ are curves defined over $K$ and birationally isomorphic over $K$. 
Then $\L(X;K) \approx \L(Y;K)$.
\end{lem}

\begin{proof}
If $X$ and $Y$ are birationally isomorphic, then there are Zariski open subsets 
$U_X \subset X$, $U_Y \subset Y$ such that $U_X \cong U_Y$ over $K$.  
Let $T$ denote the finite variety $X-U_X$.  Then
$$
\L(X;K) = \L(U_X;K) \cup \L(T;K),
$$
and $\L(T;K)$ is finite.  Therefore $\L(X;K) \approx \L(U_X;K)$, 
and similarly for $Y$, so
$$
\L(X;K) \approx \L(U_X;K) = \L(U_Y;K) \approx \L(Y;K).
$$
\end{proof}

Recall the statement of Theorem \ref{genus0}:

\begin{thmg0}
Suppose $X$ and $Y$ are irreducible 
curves over $K$, and $Y$ has genus zero.  Then ${\L}(X;K) \approx {\L}(Y;K)$ 
if and only if $X$ and $Y$ are birationally isomorphic over $K$.
\end{thmg0}

\begin{proof}[Proof of Theorem \ref{genus0}]
The `if' direction is Lemma \ref{binv}.  
Suppose now that $X$ and $Y$ are not birationally isomorphic over $K$; 
we will show that $\L(X;K) \not\approx \L(Y;K)$.  

Replacing $X$ and $Y$ by their normalizations and completions (and using 
Lemma \ref{binv} again), we may assume without loss of 
generality that $X$ and $Y$ are both smooth and projective.

\medskip\noindent{\em Case 1: $X$ has genus zero.}
In this case $X$ and $Y$ are one-dimensional Brauer-Severi varieties, 
so we can apply Proposition \ref{b-s}.
Let $c_X, c_Y \in \Br(K)[2]$ be the corresponding Brauer classes.  
Since $X$ and $Y$ are not isomorphic, there is a place $v$ such that 
(switching $X$ and $Y$ if necessary) $\inv_v(c_X) = 0$ and $\inv_v(c_Y) = 1/2$.  
Let $T$ be the (finite) set of places of $K$ different from $v$ where $\inv_v(c_X)$ and 
$\inv_v(c_Y)$ are not both zero.
If $L/K$ is a quadratic extension in which $v$ splits, but no place in $T$ splits, 
then by Proposition \ref{b-s} we have $L \in \L(X;K)$ but $L \notin \L(Y;K)$.
There are infinitely many such $L$, so $\L(X;K) \not\approx \L(Y;K)$.

\medskip\noindent{\em Case 2: $X$ has genus at least one.}
Let $K'/K$ be a finite extension large enough so that 
all $\Kb$-endomorphisms of the jacobian of $X$ are defined over $K'$, 
and $Y(K')$ is nonempty.  By Theorem \ref{static-curves} applied to $X/K'$ 
we can find infinitely many nontrivial cyclic extensions $L/K'$ 
such that $X(L) = X(K')$, so in particular $L \notin {\L}(X;K)$.
But $Y(L)$ is nonempty, so $L \in \L(Y;K)$ by Proposition \ref{b-s}.  
Since there are infinitely many such $L$, 
we conclude that $\L(X;K) \not\approx \L(Y;K)$.
\end{proof}

\subsection{Principal homogeneous spaces for abelian varieties}

The following prop\-osition was suggested by Daniel Goldstein and Zev Klagsbrun.
It shows that the answer to Question \ref{ques} is ``no'' if $\Kb$ is replaced by $K$.  
To see this, suppose that $A$ is an elliptic curve, and $\ba, \ba' \in H^1(K,A)$
generate the same cyclic subgroup,
but there is no $\alpha \in \Aut_K(A)$ such that $\ba' = \alpha\ba$.
Then the corresponding principal homogeneous spaces $X, X'$ are not isomorphic over $K$, 
but Proposition \ref{GK} shows that $\L(X;K) = \L(X';K)$.

\begin{prop}
\label{GK}
Fix an abelian variety $A$, and suppose $X$ and $X'$ are principal homogeneous 
spaces over $K$ for $A$ with corresponding classes $\ba, \ba' \in H^1(K,A)$.  
If the cyclic subgroups $\Z\ba$ and $\Z\ba'$ are equal, then $\L(X;K) = \L(X';K)$.
\end{prop}

\begin{proof}
Fix $n$ such that $n\ba = 0$.  
The short exact sequence 
$$
0 \to A[n] \to A(\Kb) \to A(\Kb) \to 0
$$ 
leads to the descent exact sequence
$$
0 \too A(K)/nA(K) \too H^1(K,A[n]) \too H^1(K,A)[n] \too 0,
$$
and it follows that $\ba$ can be represented by a cocycle $\sigma \mapsto a_\sigma$ with $a_\sigma \in A[n]$.
Since $\ba$ and $\ba'$ generate the same subgroup, for some $m \in (\Z/n\Z)^\times$ we can represent $\ba'$ by 
$\sigma \mapsto a'_\sigma$ with $a'_\sigma = ma_\sigma$.

There are isomorphisms $\phi : A \to X$, $\phi' : A \to X'$ defined over $\Kb$ 
such that if $P \in A(\Kb)$ and $\sigma \in G_K$, then 
$$
\phi(P)^\sigma = \phi(P^\sigma + a_\sigma), \qquad \phi'(P)^\sigma = \phi'(P^\sigma + a'_\sigma)
$$
In particular, if $\sigma \in G_K$ then
$$
\phi(P)^\sigma = \phi(P) \iff P^\sigma - P = -a_\sigma,
$$
so 
\begin{equation}
\label{three}
\text{$K(\phi(P))$ is the fixed field of the subgroup $\{\sigma \in G_K : P^\sigma - P = -a_\sigma\}$}
\end{equation}
and similarly with $\phi$ and $\ba$ replaced by $\phi'$ and $\ba'$.

Suppose $L \in \L(X;K)$.  Then we can fix $P \in A(\Kb)$ such that $K(\phi(P)) = L$.  
In other words, by \eqref{three} we have
\begin{equation}
\label{four}
G_L = \{\sigma \in G_K : P^\sigma - P = -a_\sigma\}.
\end{equation}
Since the set $\{P^\sigma - P + a_\sigma : \sigma \in G_K\}$ is finite and 
$m$ is relatively prime to $n$, we can 
choose $r \in \Z$ with $r \equiv m \pmod{n}$ 
such that $\{P^\sigma - P + a_\sigma : \sigma \in G_K\} \cap A[r] = 0$.  Then by \eqref{four}
\begin{multline*}
\{\sigma \in G_K : (rP)^\sigma - rP = -a'_\sigma\}
  = \{\sigma \in G_K : (rP)^\sigma - rP = -r a_\sigma\} \\
  = \{\sigma \in G_K : P^\sigma - P = -a_\sigma\} = G_L,
\end{multline*}
so \eqref{three} applied to $\phi'$ and $\ba'$ shows that $K(\phi'(rP)) = L$, i.e., $L \in \L(X';K)$.  
Thus $\L(X;K) \subset \L(X';K)$, and reversing the roles of $X$ and $X'$ shows that 
we have equality.
\end{proof}

It seems natural to ask the following question about a possible converse to Proposition \ref{GK}.

\begin{question}
\label{q2.7}
Suppose that $A$ is an abelian variety, and $X, X'$ are principal homogeneous spaces 
for $A$ over $K$ with corresponding classes $\ba, \ba' \in H^1(K,A)$.  If $\L(X;K) = \L(X';K)$, 
does it follow that $\ba$ and $\ba'$ generate the same $\End_K(A)$-submodule of $H^1(K,A)$?
\end{question}

\begin{exa}
Let $E$ be the elliptic curve 571A1 : $y^2 +y = x^3 - x^2 - 929x -10595$, 
with $\End_\Q(E) = \End_{\bar{\Q}}(E) = \Z$.
Then the Shafarevich-Tate group $\Sh(E/\Q) \cong \Z/2\Z \times \Z/2\Z$, and the three 
nontrivial elements (which generate distinct cyclic subgroups of $H^1(\Q,E)$) 
are represented by the principal homogeneous spaces
\begin{align*}
&X_1 : y^2 = -19x^4 + 112x^3 - 142x^2 - 68x - 7\\
&X_2 : y^2 = -16x^4 - 82x^3 - 52x^2 + 136x - 44\\
&X_3 : y^2 = -x^4 - 26x^3 - 148x^2 + 274x - 111.
\end{align*}
Let $d_1 = 17$, $d_2 = 41$, and $d_3 = 89$. 
A computation in Sage \cite{sage} shows that $\Q(\sqrt{d_i}) \in \L(X_j;\Q)$ if and only if $i = j$, 
so the sets $\L(X_j;\Q)$ are distinct.
\end{exa}

\section{Theorem \ref{static-avs} implies Theorem \ref{static-curves}}
\label{proof1}

In this section we deduce Theorem \ref{static-curves} from Theorem \ref{static-avs}. 
 
\begin{lem} 
\label{preimp} 
The conclusion of Theorem \ref{static-curves} depends 
only on the birational equivalence class of  $X$ over $K$.  More precisely, if 
$X$, $Y$ are irreducible curves over $K$, birationally isomorphic over $K$, and
$\ell$ is sufficiently large (depending on $X$ and $Y$), 
then 
$$
\text{$X$ is $\ell$-diophantine-stable over $K$ $\iff$ $Y$ is $\ell$-diophantine-stable over $K$}.
$$
\end{lem}  

\begin{proof}  
It suffices to prove the lemma in the case that $Y$ is a dense open subset of $X$. 
This is because any two $K$-birationally equivalent curves contain a common open dense 
subvariety.

Let $T := X-Y$.  Then $T = \coprod_{i \in I} \Spec(K_i)$ for some finite index set $I$ and
number fields $K_i$ containing $K$. 
Let $\delta = \max\{[K_i:K] : i \in I\}$.  
Then for every cyclic extension $L/K$ of prime-power degree $\ell^n$ with $\ell > \delta$, 
we have $L \cap K_i = K$ for all $i \in I$, so $T(L) = T(K)$ and  
$
X(L) = X(K) \iff Y(L) = Y(K).
$
\end{proof}

It suffices, then, to prove Theorem \ref{static-curves} for irreducible projective 
smooth curves $X$.

\begin{lem} 
\label{preimp2} 
Suppose $f:X \to Y$ is a nonconstant map (defined over $K$) of irreducible curves over $K$.  
If $\ell$ is sufficiently large (depending on $X$, $Y$, and $f$), and
$Y$ is $\ell$-diophantine-stable over $K$, then $X$ is $\ell$-diophantine-stable over $K$.
\end{lem}

\begin {proof} 
By Lemma \ref{preimp} we may assume that  $f:X \to Y$  is a morphism of finite degree, 
say $d$, of smooth projective curves. Let $L/K$ be a cyclic extension of degree 
$\ell^n$ with $\ell > d$ such that $Y(L) = Y(K)$.  We will show that $X(L) = X(K)$.

Consider a point $x\in X(L)$, and let $y := f(x) \in Y(L)=Y(K)$.  Form the fiber, 
i.e., the zero-dimensional scheme $T:=f^{-1}(y)$. Then $x\in T(L)$.
As in the proof of Lemma \ref{preimp}, the reduction of the scheme $T$ is a disjoint union of 
spectra of number fields of degree at most $d$ over $K$.
Since $\ell > d$, we have $T(L)=T(K)$ and hence $ x\in X(K)$.
\end{proof}
   
\begin{lem}
\label{lemimp}   
Theorem \ref{static-avs}  $\implies$ Theorem \ref{static-curves}.
\end{lem}
   
\begin{proof}  
Let $\tilde{X}$ be the completion and normalization of $X$.  
Let $D$ be a $K$-rational divisor on $\tilde{X}$ of nonzero degree $d$, 
and define a nonconstant map over $K$ from $\tilde{X}$ to its 
jacobian $J(\tilde{X})$ by $x \mapsto D - d \cdot [x]$.  Let $A$ be a simple 
abelian variety quotient of $J(\tilde{X})$ defined over $K$, and 
let $Y \subset A$ be  the image 
of $\tilde{X}$.  Theorem \ref{static-avs} applied to $A$ shows that 
there is a set $S$ of primes, with positive density, such that $A$ (and hence $Y$ as well) 
is $\ell$-diophantine-stable 
over $K$ for every $\ell\in S$.  It follows from Lemmas \ref{preimp} and \ref{preimp2} that 
(for $\ell$ sufficiently large) $X$ is $\ell$-diophantine-stable over $K$ for every $\ell\in S$ as well, 
i.e., the conclusion of Theorem \ref{static-curves} holds for $X$.
\end{proof}

\section{Infinite extensions}
\label{uncpf}

In this section we will prove Corollaries \ref{uncount} and \ref{shlap}.

\begin{thm}
\label{unco}
Suppose $V$ is either a simple abelian variety over $K$ as in Theorem \ref{static-avs} 
or an irreducible curve over $K$ as in Theorem \ref{static-curves}.  For every finite 
set $\Sigma$ of places of $K$, there are uncountably many pairwise non-isomorphic 
extensions $L$ of $K$ in $\Kb$ such that all places in $\Sigma$ split completely in $L$, 
and $V(L) = V(K)$.
\end{thm}

\begin{proof} 
Let 
$$
{\mathcal N}:= (n_1, n_2, n_3, \dots)
$$ 
be an arbitrary infinite sequence of positive integers. 
Using Theorem \ref{static-curves}, choose a prime $\ell_1$ and a Galois extension $K_1/K$, 
completely split at all $v\in\Sigma$, that 
is cyclic of degree $\ell_1^{n_1}$ and such that $V(K_1) = V(K)$.  
Continue inductively, using Theorem \ref{static-curves}, to choose an increasing 
sequence of primes $\ell_1 < \ell_2 < \ell_3 < \cdots$ and a tower of fields 
$K \subset K_1 \subset K_2 \subset K_3 \subset \cdots$ such that $K_i/K_{i-1}$ 
is cyclic of degree $\ell_i^{n_i}$, completely split at all places above $\Sigma$, and 
$X(K_i) = X(K)$ for every $i$.
Let $K_{\mathcal N}:=  \cup_{i\ge 1}K_i \subset \Kb \cap K_v$.

We have that $X(K_{\mathcal N}) = X(K)$ for every $\mathcal N$.  We claim further that 
no matter what choices are made for the $\ell_i$, the construction  
$$
{\mathcal N} \mapsto K_{\mathcal N}
$$ 
establishes an {\it injection} of the (uncountable) set of sequences ${\mathcal N}$ 
of positive integers into the set of subfields of $\Kb \cap K_v$.  
To see this, observe that 
by writing a subfield $F \subset \Kb$ as a union of finite extensions of $K$, 
one can define the degree $[F:\Q]$ as a formal product 
$\prod_p p^{a_p}$ over all primes $p$, with $a_p \le \infty$ (i.e., a supernatural number).  Then 
$[K_{\mathcal N}:K] = \prod_i \ell_i^{n_i}$, and since the $\ell_i$ 
are increasing, this formal product determines the sequence $\mathcal N$.
Therefore there are uncountably many such fields  $K_{\mathcal N}$, and they are pairwise 
non-isomorphic.
\end{proof}

Recall the statement of Corollary \ref{uncount}:

\begin{coruncount}  
Let $p\ge 23$ and $p\ne 37, 43, 67, 163$. There are uncountably many 
pairwise non-isomorphic subfields $L$ of $\bar{\Q}$ such that no elliptic 
curve defined over $L$ possesses an $L$-rational subgroup 
of order $p$.
\end{coruncount}

\begin{proof}[Proof of Corollary \ref{uncount}]
By \cite{M}, if $p$ is a prime satisfying the hypotheses of the corollary, then the 
modular curve  $X:=X_0(p)$  defined over ${\Q}$ only has two rational points, 
namely the cusps $\{0\}$ and $\{\infty\}$, and the genus of $X$ is greater than zero.  
Since the jacobian of $X$ is 
semistable, its endomorphisms are all defined over $\Q$ (see \cite{R}).  
Thus the hypotheses of Theorem \ref{static-curves} hold with $K := \Q$,
and Theorem \ref{unco} produces uncountably many subfields $L$ of $\bar{\Q}$ 
such that $X_0(p)$ has no non-cuspidal $L$-rational points.
\end{proof}

\begin{cor}
\label{pscor}
For every prime $p$, there are uncountably many pairwise non-isomorphic 
fields $L \subset \bar{\Q}$ such that
\begin{enumerate}
\item
$L$ is totally real,
\item
$p$ splits completely in $L$,
\item
there is an elliptic curve $E$ over $\Q$ 
such that $E(L)$ is a finitely generated infinite group.
\end{enumerate}
\end{cor}

\begin{proof}
Fix any elliptic curve $E$ over $\Q$ with positive rank, and without complex multiplication.
Apply Theorem \ref{unco} to $E$ with $\Sigma = \{\infty,p\}$.
\end{proof}

Recall the statement of Corollary \ref{shlap}:

\begin{corshlap}
For every prime $p$, there are uncountably many pairwise non-isomorphic 
totally real fields $L$ of 
algebraic numbers in $\Qp$ over which the following two statements both hold:
\begin{enumerate} 
\item 
There is a diophantine definition of ${\Z}$ in the ring of integers $\O_L$ of $L$. 
In particular, Hilbert's Tenth Problem has a negative answer for $\O_L$; 
i.e., there does not exist an algorithm to determine whether a polynomial 
(in many variables) with coefficients in $\O_L$ has a solution in $\O_L$.
 \item 
There exists a first-order definition of the ring $\Z$ in $L$. 
The first-order theory for such fields $L$ is undecidable.
\end{enumerate}
\end{corshlap}

\begin{proof}[Proof of Corollary \ref{shlap}]
The corollary follows directly from Corollary \ref{pscor} and results of Shlap\-entokh, 
as follows.
Suppose $L$ is an infinite extension of $\Q$ satisfying Corollary \ref{pscor}(i,ii,iii).
Assertion (i) follows from Corollary \ref{pscor}(i,iii) and \cite[Main Theorem A]{Shlapentokh}.  
Since $p$ splits completely in $L$, the prime $p$ is $q$-bounded (for every rational prime $q$) 
in the sense of \cite[Definition 4.2]{Shlap2}, so assertion (ii) follows from 
Corollary \ref{pscor}(ii,iii) and \cite[Theorem 8.5]{Shlap2}.
\end{proof}

\part{Abelian varieties and diophantine stability}
\label{part2}

\section{Strategy of the proof}
\label{intro2}

\begin{notation}
For sections \ref{twists} through \ref{hyps} fix a simple abelian variety $A$ defined over 
an arbitrary field $K$ (in practice $K$ will be a number field or one of its completions).
Let $\OO$ denote the center of $\End_K(A)$, and $\M := \OO \otimes \Q$.  Since $A$ is simple, 
$\M$ is a number field and $\OO$ is an order in $\M$.
Fix a rational prime $\ell$ that does not divide the discriminant of $\OO$, and 
fix a prime $\lambda$ of $\M$ above $\ell$.  In particular $\ell$ is unramified in $\M/\Q$.  
Denote by $\M_\lambda$ the completion of $\M$ at $\lambda$.
\end{notation}

In the following sections we develop the machinery that we need to prove Theorem \ref{static-avs}.
Here is a description of the strategy of the proof.

The standard method---perhaps the only fully proved method---of finding upper bounds for 
Mordell-Weil ranks is the {\em method of descent} that seems to have been already present 
in some arguments due to Fermat and has been elaborated and refined ever since.  
These days ``descent" is done via computation of  {\em Selmer groups}. 
To check for diophantine stability we will be considering the relative theory; that is, 
how things change when passing from our base field $K$  to $L$, a cyclic extension 
of prime power degree $\ell^n$ over $K$.   
The Galois group $\Gal(L/K)$ acts on the finite dimensional $\Q$-vector space $A(L)\otimes \Q$. 
Diophantine stability here requires that the action be trivial; i.e, it requires that for any 
Galois character $\chi_i:G_K \to {\C}^*$ of order $\ell^i$  ($0<i\le n$) that cuts out a 
nontrivial sub-extension, $L_i/K$ of $L/K$, the $\chi_i$-component of the 
$\Gal(L/K)$-representation $A(L)\otimes {\C}$ vanishes.  Since this representation is 
defined over ${\Q}$, if, for $i >0$,  the $\chi_i$-part of  $A(L)\otimes {\C}$ vanishes then 
\begin{equation}
\label{i1}
A(L_i)\otimes {\Q} =A(L_{i-1})\otimes {\Q}.
\end{equation}

Sections \ref{twists} and \ref{lflc} below prepare for a discussion of a certain relevant 
relative Selmer group, denoted $\Sel(L_i/K, A[\lambda])$ defined in Section \ref{sgss} 
that has the property that  its vanishing implies \eqref{i1}.  More precisely, 
Proposition \ref{tower} below gives:
$$
\rk_\Z A(L) \le \rk_\Z A(K) + 
   \rk_{\Z}(\OO)\sum_{i=1}^n\phi(\ell^i)\cdot \dim_{\OO/\lambda}\Sel(L_i/K, A[\lambda]).
$$
The key to the technique we adopt is that for all cyclic $\ell^n$-extensions 
$L/K$ (for fixed $\ell$), the corresponding relative Selmer groups $\Sel(L/K, A[\lambda])$ are
canonically `tied together' as finite dimensional subspaces of a single (infinite dimensional)  
$\OO/\lambda$-vector space, namely $H^1(G_K, A[\lambda])$. 
The subspace $\Sel(L/K, A[\lambda])$ of $H^1(G_K, A[\lambda])$ is determined  by specific local 
conditions at all places $v$ of $K$, these local conditions in turn being determined by 
$A/K_v$ and $L_v/K_v$ where $L_v$ is  the completion of $L$ at any prime of $L$ above $v$. 
Even more specifically, $\Sel(L/K, A[\lambda])$ is determined by $A/K$ and the collection of 
local extensions $L_v/K_v$  for $v$ primes of $K$; moreover, an `artificial Selmer subgroup' 
of $H^1(G_K, A[\lambda])$ can be defined corresponding to any collection of local extensions 
$L_v/K_v$ even if this collection doesn't come from a global $L/K$.
 
Nevertheless, when passing from one global extension $L/K$ to another $L'/K$ 
of the same degree, one needs only change the local conditions that determine $\Sel(L/K, A[\lambda])$ 
at a finite set of primes $S$ to obtain the local conditions that 
determine $\Sel(L'/K, A[\lambda])$.  Our aim, of course, is to find a large quantity of 
extensions  $L/K$ with $\Sel(L/K, A[\lambda]) = 0$.  We do this by starting with an arbitrary $L/K$
and then constructing inductively appropriate finite sets $\Sigma$, 
with changes of local conditions at the primes in $\Sigma$ 
corresponding to extensions $L'/K$ such that the $\Sel(L_i'/K, A[\lambda]) = 0$ for all $i$. 
 
For this, it is essential that we are supplied with what we call {\em critical primes} 
and {\em silent primes}.
 
\medskip\noindent
{\em Enough critical primes:} Critical primes are judiciously chosen primes $v$ for which a 
change of local condition at $v$ lowers the dimension of the corresponding Selmer group by $1$. 
They are primes $v$ of good reduction for $A$ and such that  $\ell$ divides the order of the 
multiplicative group of the residue field of $v$  (no problem finding primes of this sort) 
and such that the action of the Frobenius element at $v$ on the vector space $A[\lambda]$ has a 
one-dimensional fixed space. Here---given some other hypotheses that will obtain when 
$\ell \gg 0$---we make use of global duality to guarantee that between the strictest local 
condition at $v$ and the most relaxed local condition at $v$, the corresponding 
Selmer groups differ in size by one dimension.  Moreover, we  engineer our choice of prime 
$v$ so that the localization map from $\Sel(L/K, A[\lambda])$ onto the one-dimensional  
Selmer local condition at $v$ is surjective. In this set-up, any change of local condition 
subgroup at $v$ will define an `artificial global Selmer group' of dimension 
$\dim_{\OO/\lambda} \Sel(L/K, A[\lambda]) - 1$. 
            
Iterating this process a finite number of times leads us to a modification of the 
initial local conditions at finitely many critical primes, such that the artificially 
constructed Selmer group is zero.  This proved in Proposition \ref{l7.13}.
           
\medskip\noindent
{\em Enough silent primes:}  
For $\ell \gg 0$, silent primes are primes $v$ of good reduction 
for $A$ such that $\ell$ divides the order of the multiplicative group of the 
residue field of $v$, and such that the Frobenius element at $v$ has {\em no} 
nonzero fixed vectors in its action on $A[\lambda]$.  For these primes the local
cohomology group vanishes, so changing the local extension $L'_v/K_v$ at such primes 
doesn't change the local condition, hence doesn't change the Selmer group. 
By making use of silent primes, we can ensure that we have infinitely many 
collections of local data such that the corresponding (artificial) Selmer group 
is zero.  In addition, Larsen in his appendix requires the existence of silent primes 
in order to prove the existence of critical primes.

\medskip
In the description above, we chose a finite collection of local extensions 
$L'_v/K_v$ with specified properties 
for the construction of our Selmer group, a single place $v$ at a time, 
to keep lowering dimension. At the end of this process, we need to have a {\em global} 
extension $L'/K$ corresponding to our collection of local extensions $\{L'_v/K_v\}_v$. 
The existence of such an $L'$ is given by Lemma \ref{l7.15}. 
           
In the appendix, Michael Larsen proves a general theorem (Theorem \ref{main}) 
guaranteeing the existence of sufficiently many critical and silent primes in 
the general context of Galois representations  on $A[\lambda]$ for $A$ a simple 
abelian variety over a number field.

\section{Twists of abelian varieties}
\label{twists}

Keep the notation from from the beginning of \S\ref{intro2}.
In this section we recall results from \cite{mrs} about twists of abelian varieties.  
We will use these twists in \S\ref{lflc} and \S\ref{sgss} to define the relative Selmer groups 
$\Sel(L/K,A[\lambda])$ described in \S\ref{intro2}.

Fix for this section 
a cyclic extension $L/K$ of degree $\ell^n$ with $n \ge 0$.  Let $G := \Gal(L/K)$.
If $n \ge 1$ (i.e., if $L \ne K$), let $L'$ be the (unique) subfield of $L$ of
degree $\ell^{n-1}$ over $K$ and $G' := \Gal(L'/K) = G/G^{\ell^{n-1}}$.

\begin{defn}
\label{twistdef}
Define an ideal $\I_{L} \subset \OO[G]$ by
$$
\I_{L} := 
\begin{cases}
\ker(\OO[G] \too \OO[G']) & \text{if $n \ge 1$}, \\
\OO[G] & \text{if $n = 0$}.
\end{cases}
$$
Then $\rk_\OO(\I_{L}) = \varphi(\ell^n)$, where $\varphi$ is the Euler $\varphi$-function, and 
we define the {\em $L/K$-twist} $A_{L}$ of $A$ to be the abelian variety 
$\I_{L} \otimes A$ of dimension $\varphi(\ell^n)\dim(A)$ 
as defined in \cite[Definition 1.1]{mrs}.  Concretely, if $n \ge 1$ then
$$
A_{L} := \ker(\Res^L_K A \too \Res^{L'}_K A).
$$
Here $\Res^L_K A$ denotes the Weil restriction of scalars of $A$ from $L$ to $K$, 
and the map is obtained by identifying
$\Res^L_K A = \Res^{L'}_K \Res^L_{L'} A$ and using the canonical map $\Res^L_{L'} A \to A$.
If $n = 0$, we simply have $A_{K} = A$.
\end{defn}

See \cite[\S3]{alc} or \cite{mrs} for a discussion of $A_{L}$ and its properties.

\begin{defn}
\label{anotherdef}
With notation as above, 
let $\bN_{L/L'} := \sum_{\sigma\in\Gal(L/L')}\sigma \in \OO[G]$ if $n \ge 1$ and 
$\bN_{L/L'} = 0$ if $n = 0$, and define 
$$
R_{L} := \OO[G]/\bN_{L/L'}\OO[G]
$$
so $\rk_\OO R_{L} = \varphi(\ell^n)$.

Fixing an identification $G \isom \bmu_{\ell^n}$ of $G$ with 
the group of $\ell^n$-th roots of unity in $\bar{\M}$ induces an inclusion 
$$
R_{L} \hookto \M(\bmu_{\ell^n})
$$
that identifies $R_{L}$ with an order in $\M(\bmu_{\ell^n})$.  
Since $\ell$ is unramified in $\M/\Q$ we have that $\lambda$ is totally ramified in 
$\M(\bmu_{\ell^n})/\M$, and we let 
$\lambda_{L}$ denote the (unique) prime of $R_{L}$ above $\lambda$.
\end{defn}

Note that $\I_{L}$ is the annihilator of $\bN_{L/L'}$ in $\OO[G]$, 
so $\I_{L}$ is an $R_{L}$-module.
The following proposition summarizes some of the properties of $A_{L}$ 
proved in \cite{mrs} that we will need.

\begin{prop}
\label{ros}
\begin{enumerate}
\item
The natural action of $G$ on $\Res^L_K(A)$ 
induces an inclusion $R_{L} \subset \End_K(A_{L})$.
\item
For every commutative $K$-algebra $D$, and every Galois extension $F$ of $K$ containing $L$, there 
is a natural $R_{L}[\Gal(F/K)]$-equivariant isomorphism 
$$
A_{L}(D \otimes_K F) \cong \I_{L} \otimes_\OO A(D \otimes_K F),
$$
where $R_{L}$ acts on $A_{L}$ via the inclusion of (i) and
on $\I_{L} \otimes A(D \otimes_K F)$ by multiplication on $\I_{L}$, 
and $\gamma \in \Gal(L/K)$ acts on $\I_{L} \otimes A(D \otimes_K F)$ as $\gamma^{-1} \otimes (1 \otimes\gamma)$.
\item
For every ideal $\b$ of $\OO$, the isomorphism of (ii) induces an isomorphism of $R_{L}[G_K]$-modules
$$
A_{L}[\b] \cong \I_{L} \otimes_\OO A[\b].
$$
\item
For every commutative $K$-algebra $D$, the isomorphism of (ii) induces an isomorphism 
of $\OO$-modules
$$
A_{L}(D) \cong \I_{L} \otimes_{\OO[G]} A(D \otimes_K L)
$$
where $\gamma \in \Gal(L/K)$ acts on $D \otimes_K L$ as $1 \otimes \gamma$.
\end{enumerate}
\end{prop}

\begin{proof}
The first assertion is \cite[Theorem 5.5]{mrs}, and the second is \cite[Lemma 1.3]{mrs}.
Then (iii) follows from (ii) by taking $D := K$ and $F := \Kb$ (see \cite[Theorem 2.2]{mrs}), 
and (iv) follows from (ii) by setting $F := L$ and taking $\Gal(L/K)$ invariants of both sides 
(see \cite[Theorem 1.4]{mrs}).
\end{proof}

\begin{cor}
\label{roscor}
The isomorphism of Proposition \ref{ros}(iii) induces an isomorphism of $\OO[G_K]$-modules
$$
A_{L}[\lambda_{L}] \cong A[\lambda].
$$
\end{cor}

\begin{proof}
Fix a generator $\gamma$ of $G$, and let $\bar\gamma$ denote its projection to $R_{L}$.
Then $\lambda_{L}$ is generated by $\lambda$ and $\bar\gamma-1$, so Proposition \ref{ros}(iii) shows that
$$
A_{L}[\lambda_{L}] = A_{L}[\lambda][\bar\gamma - 1] = (\I_{L} \otimes A[\lambda])[\bar\gamma-1].
$$
If $L = K$ there is nothing to prove.  If $L \ne K$ then $\I_{L}$ is defined by the exact sequence
\begin{equation}
\label{e5.5}
0 \too \I_{L} \too \OO[G] \too \OO[G'] \too 0.
\end{equation}
Tensoring the free $\OO$-modules of \eqref{e5.5} with $A[\lambda]$ and taking the kernel of $\gamma-1$ gives
\begin{equation}
\label{e5.6}
0 \too A_{L}[\lambda_{L}] \too (\OO[G]\otimes A[\lambda])[\gamma-1] \too \OO[G']\otimes A[\lambda].
\end{equation}
Explicitly, 
$$
\textstyle
(\OO[G]\otimes A[\lambda])[\gamma-1] = \{\sum_{g \in G} g \otimes a : a \in A[\lambda]\} \cong A[\lambda],
$$
and this is in the kernel of the right-hand map of \eqref{e5.6},
so the corollary follows.
\end{proof}

\section{Local fields and local conditions}
\label{lflc}

In this section we use the twists $A_L$ of \S\ref{twists} to define the local conditions 
that will be used in \S\ref{sgss} to define our relative Selmer groups $\Sel(L/K,A[\lambda])$.

Let $A$, $\OO$, $\ell$, and $\lambda$ be as in \S\ref{twists}, and keep the rest of the 
notation of \S\ref{intro2} and \S\ref{twists} as well.  For this section we 
restrict to the case where $K$ is a local field of characteristic zero, 
i.e., a finite extension of some $\Q_\ell$ or of $\R$.
Fix for this section a cyclic extension $L/K$ of $\ell$-power degree, and let $G := \Gal(L/K)$.

\begin{defn}
\label{localdef}
Define 
$\H_{\lambda}(L/K) \subset H^1(K,A[\lambda])$ to be the image of the composition
$$
A_{L}(K)/\lambda_{L} A_{L}(K) \hookto  H^1(K,A_{L}[\lambda_{L}]) \cong H^1(K,A[\lambda])
$$
where $\lambda_{L}$ is as in Definition \ref{anotherdef}, 
the first map is the Kummer map, and the second map is the isomorphism of 
Corollary \ref{roscor}.
(This Kummer map depends on the choice of a generator of $\lambda_{L}/\lambda_{L}^2$, but its image 
is independent of this choice.)  When $L= K$, $\H_{\lambda}(K/K)$ is just the image of the 
Kummer map
$$
A(K)/\lambda A(K) \hookto H^1(K,A[\lambda])
$$
and we will denote it simply by $\H_{\lambda}(K)$.
We suppress the dependence on $A$ from the notation when possible, since $A$ is fixed throughout this section.
\end{defn}

If $K$ is nonarchimedean of characteristic different from $\ell$, and $A/K$ 
has good reduction, we define
$$
\Hu(K,A[\lambda]) := H^1(K^\ur/K,A[\lambda]),
$$
the unramified subgroup of $H^1(K,A[\lambda])$.

\begin{lem}
\label{xlem}
Suppose $K$ is nonarchimedean of residue characteristic different from $\ell$.
\begin{enumerate}
\item 
We have $\dim_{\Fl}(\H_{\lambda}(L/K)) = \dim_{\Fl}A(K)[\lambda]$.
\item
If $A$ has good reduction and $\phi \in G_K$ is 
an automorphism that restricts to Frobenius in $\Gal(K^\ur/K)$, then
$$
\dim_{\Fl}(\H_{\lambda}(L/K)) = \dim_{\Fl}A[\lambda]/(\phi-1)A[\lambda].
$$
\end{enumerate}
\end{lem}

\begin{proof}
Suppose $K$ is nonarchimedean of residue characteristic different from $\ell$.  Then 
$A_L(K)$ has a subgroup of finite index that is $\ell$-divisible, so
$$
A_L(K)/\lambda_L A_L(K) \cong A_L(K)_{\tors}/\lambda_L A_L(K)_{\tors} 
   \cong A_L(K)[\lambda_L] \cong A(K)[\lambda]
$$
where the second isomorphism is non-canonical and the third is Corollary \ref{roscor}.
Since $\H_{\lambda}(L/K) \cong A_L(K)/\lambda_L A_L(K)$ by definition, this proves (i).

If further $A$ has good reduction then $A[\lambda] \subset A(K^\ur)$.  If $\phi$ is 
an Frobenius automorphism in $\Gal(K^\ur/K)$, then $A(K)[\lambda] = A[\lambda]^{\phi=1}$
so
$$
\dim_{\Fl}A(K)[\lambda] = \dim_{\Fl}A[\lambda]^{\phi=1} = \dim_{\Fl}(A[\lambda]/(\phi-1)A[\lambda].
$$
Now (ii) follows from (i).
\end{proof}

\begin{lem}
\label{urrem}
Suppose $K$ is nonarchimedean of residue characteristic different from $\ell$, $A/K$ 
has good reduction, and $L/K$ is unramified.
\begin{enumerate}
\item
If $\phi \in G_K$ is 
an automorphism that restricts to Frobenius in $\Gal(K^\ur/K)$, then evaluation 
of cocycles at $\phi$ induces an isomorphism
$$
\Hu(K,A[\lambda]) \isom A[\lambda]/(\phi-1)A[\lambda].
$$
\item
The twist $A_{L}$ has good reduction, and $\H_{\lambda}(L/K) = \Hu(K,A[\lambda])$.  
In particular under these assumptions $\H_{\lambda}(L/K)$ is independent of $L$.
\end{enumerate}
\end{lem}

\begin{proof}
This is well-known.  For (i), see for example \cite[Lemma 1.3.2(i)]{eulersys}.
That $A_{L}$ has good reduction when $L/K$ is unramified follows from 
the criterion of N{\'e}ron-Ogg-Shafarevich and Proposition \ref{ros}(iii).  
Since $A_{L}$ has good reduction and $L/K$ is unramified, 
we have $\H_\lambda(L/K) \subset \Hu(K,A[\lambda])$, and further
$$
\dim_{\Fl}\H_\lambda(L/K) = \dim_{\Fl}(A[\lambda]/(\phi-1)A[\lambda])
   = \dim_{\Fl}\Hu(K,A[\lambda])
$$
using Lemma \ref{xlem}(ii) for the first equality, and (i) for the second.  This proves (ii).
\end{proof}

\begin{lem}
\label{ros2}
Suppose $K$ is nonarchimedean of residue characteristic 
different from $\ell$, $A/K$ has good reduction, and $L/K$ is nontrivial and 
totally ramified.  Let $L_1$ be the unique cyclic extension of $K$ of degree $\ell$ 
in $L$.  Then 
the map 
$$
A_{L}(K)/\lambda_{L} A_{L}(K) \to A_{L}(L_1)/\lambda_{L} A_{L}(L_1)
$$
induced by the inclusion $A_{L}(K) \subset A_{L}(L_1)$ is the zero map.
\end{lem}

\begin{proof}
Since $A/K$ has good reduction and the residue characteristic is different from $\ell$, 
we have that $K(A[\ell^\infty])/K$ is unramified.  
Since $L/K$ is totally ramified, $L \cap K(A[\ell^\infty]) = K$.  Hence 
$A(L)[\ell^\infty] = A(K)[\ell^\infty]$, so by Proposition \ref{ros}(iii), 
\begin{multline}
\label{invptor0}
A_{L}(K)[\ell^\infty] = (\I_{L} \otimes A[\ell^\infty])^{G_K}
   = ((\I_{L} \otimes A[\ell^\infty])^{G_L})^{G} \\
   = (\I_{L} \otimes (A(L)[\ell^\infty]))^G
   = (\I_{L} \otimes (A(K)[\ell^\infty]))^G.
\end{multline}
As in the proof of Corollary \ref{roscor}, tensoring the exact sequence \eqref{e5.5}
with $A(K)[\ell^\infty]$ and taking $G$ invariants gives an exact sequence
$$
0 \too (\I_{L} \otimes A(K)[\ell^\infty])^G \too (\OO[G] \otimes A(K)[\ell^\infty])^G
   \too (\OO[G'] \otimes A(K)[\ell^\infty])^G.
$$ 
Since $G$ acts trivially on $A(K)[\ell^\infty]$, we have
$$
(\OO[G] \otimes A(K)[\ell^\infty])^G = \{\textstyle\sum_{g\in G} g \otimes a : a \in A(K)[\ell^\infty]\}.
$$
The map to $\OO[G'] \otimes A(K)[\ell^\infty]$ sends $\sum_{g\in G} g \otimes a$ to 
$\ell\sum_{g \in G'} g \otimes a$, which is zero if and only if $a \in A[\ell]$.
Therefore
$$
(\I_{L} \otimes (A(K)[\ell^\infty]))^G = \{\textstyle\sum_{g \in G} g \otimes a : a \in A(K)[\ell]\},
$$
and combining this with \eqref{invptor0} gives
\begin{equation}
\label{invptor}
A_{L}(K)[\ell^\infty] = \{\textstyle\sum_{g \in G} g \otimes a : a \in A(K)[\ell]\}.
\end{equation}
An identical calculation shows that
\begin{equation}
\label{invptor3}
A_{L}(L_1)[\ell] = \{\textstyle\sum_{i = 0}^{\ell^n-1} (\gamma^i \otimes a_i) : 
  \text{$a_i \in A(K)[\ell]$ and $a_i = a_j$ if $i \equiv j \hskip -5pt\pmod{\ell}$}\}.
\end{equation}
If $a \in A(K)[\ell]$, then using the identification \eqref{invptor3} 
we have $\sum_{i=0}^{\ell^n-1}(\gamma^i \otimes ia) \in A_{L}(L_1)[\ell]$, and
$$
(\gamma-1)\sum_{i=0}^{\ell^n-1}(\gamma^i \otimes ia) = -\sum_{i=0}^{\ell^n-1}\gamma^i \otimes a.
$$
Taken together with \eqref{invptor}, this proves that
$$
A_{L}(K)[\ell^\infty] \subset (\gamma-1) A_{L}(L_1) \subset \lambda_L A_{L}(L_1)
$$
Now the lemma follows, because the map 
$$
A_{L}(K)[\ell^\infty] \onto A_{L}(K)/\lambda_{L} A_{L}(K)
$$
is surjective (since the residue characteristic of $K$ is different from $\ell)$.
\end{proof}

\begin{prop}
\label{gl}
Suppose $A/K$ has good reduction, $K$ is nonarchimedean of resi\-due characteristic 
different from $\ell$, and $L/K$ is nontrivial and totally ramified.
\begin{enumerate}
\item
If $K \subsetneq L' \subseteq L$ then $\H_{\lambda}(L'/K) = \H_{\lambda}(L/K)$.
\item
$\Hu(K,A[\lambda]) \cap \H_{\lambda}(L/K) = 0$.
\end{enumerate}
\end{prop}

\begin{proof}
Let $L_1$ be the cyclic extension of $K$ of degree $\ell$ in $L$.  In the 
commutative diagram
$$
\xymatrix@R=15pt{
A_{L}(L_1)/\lambda_{L} A_{L}(L_1) \ar@{^(->}[r] 
   & H^1(L_1,A_{L}[\lambda_{L}]) \ar^-{\sim}[r] & H^1(L_1,A[\lambda]) \\
A_{L}(K)/\lambda_{L} A_{L}(K) \ar[u] \ar@{^(->}[r] 
   & H^1(K,A_{L}[\lambda_{L}]) \ar[u] \ar^-{\sim}[r] & H^1(K,A[\lambda]) \ar[u]
}
$$
the left-hand vertical map is zero by Lemma \ref{ros2}, so by definition of $\H_{\lambda}(L/K)$
we have
\begin{equation}
\label{hk}
\H_{\lambda}(L/K) \subset \ker(H^1(K,A[\lambda]) \to H^1(L_1,A[\lambda])).
\end{equation}
Since the inertia group acts trivially on $A[\lambda]$, we have 
$A[\lambda]^{G_L} = A[\lambda]^{G_{L_1}} = A[\lambda]^{G_{K}}$, so
\begin{multline}
\label{kd}
\ker(H^1(K,A[\lambda]) \to H^1(L_1,A[\lambda])) = H^1(L_1/K,A[\lambda]^{G_{L_1}}) \\
   = H^1(L_1/K,A[\lambda]^{G_{K}}) 
   = \Hom(\Gal(L_1/K),A(K)[\lambda]).
\end{multline}
We have (using Lemma \ref{xlem}(i) for the first equality)
\begin{equation}
\label{kd2}
\dim_{\Fl}\H_{\lambda}(L/K) = \dim_{\Fl}A(K)[\lambda]
   = \dim_{\Fl}\Hom(\Gal(L_1/K),A(K)[\lambda]).
\end{equation}
Combining \eqref{hk}, \eqref{kd}, and \eqref{kd2} shows that the inclusion 
\eqref{hk} must be an equality.  This proves (i), because the kernel in \eqref{hk} 
depends only on $L_1$.
Assertion (ii) follows from \eqref{hk} and the fact that (since $L_1/K$ is 
totally ramified) the restriction map 
$$
\Hu(K,A[\lambda]) \hookto \Hu(L_1,A[\lambda]) \subset H^1(L_1,A[\lambda])
$$ 
is injective.
\end{proof}

\begin{rem}
The proof of Proposition \ref{gl} shows that if $A$ has good reduction, and $L/K$ is a 
ramified cyclic extension of degree $\ell$, then $\H_{\lambda}(L/K)$ is the ``$L$-transverse'' 
subgroup of $H^1(K,A[\lambda])$, as defined in \cite[Definition 1.1.6]{kolysys}.
\end{rem}

\section{Selmer groups and Selmer structures}
\label{sgss}

In this section we use the definitions of \S\ref{twists} and \S\ref{lflc} to 
define the relative Selmer groups $\Sel(L/K,A[\lambda])$ described in \S\ref{intro2}.

Keep the notation of the previous sections, except that from now on $K$ is a number field.  
If $v$ is a place of $K$ we will denote by $L_v$ the completion of $L$ at some 
fixed place above $v$.  We will write $A_L$, $R_L$, $\I_L$, and $\lambda_L$ for the 
objects defined in \S\ref{twists} using the extension $L/K$, and 
$A_{L_v}$, $R_{L_v}$, $\I_{L_v}$, and $\lambda_{L_v}$ for the ones corresponding 
to the extension $L_v/K_v$.

\begin{defn}
\label{globalsd}
If $L/K$ is a cyclic extension of $\ell$-power degree, we define the $\lambda$-Selmer group 
$\Sel(L/K,A[\lambda]) \subset H^1(K,A[\lambda])$ by
$$
\Sel(L/K,A[\lambda]) := \{c \in H^1(K,A[\lambda]) : \text{$\loc_v(c) \in \H_{\lambda}(L_v/K_v)$ for every $v$}\}.
$$
Here $\loc_v : H^1(K,A[\lambda]) \to H^1(K_v,A[\lambda])$ is the localization map,  
$K_v$ is the completion of $K$ at $v$, and $L_v$ is the completion of $L$ at any place above $v$.
When $L = K$ this is the standard $\lambda$-Selmer group of $A/K$, and we denote it by 
$\Sel(K,A[\lambda])$.
\end{defn}

\begin{rem}
The Selmer group $\Sel(L/K,A[\lambda])$ defined above consists of all classes 
$c \in H^1(K,A[\lambda])$ such that for every $v$, the localization $\loc_v(c)$ lies in 
the image of the composition of the upper two maps in the diagram
\begin{equation}
\label{classsel}
\raisebox{35pt}{
\xymatrix@R=15pt{
A_{L_v}(K_v)/\lambda_{L_v}A_{L_v}(K_v) \ar@{^(->}[r] & H^1(K_v,A_{L_v}[\lambda_{L_v}]) 
   \ar_{\cong}[d] \\
& H^1(K_v,A[\lambda]) \\
A_{L}(K_v)/\lambda_{L}A_{L}(K_v) \ar@{^(->}[r] & H^1(K_v,A_{L}[\lambda_{L}]) 
   \ar^{\cong}[u]
}}
\end{equation}
On the other hand, the classical 
$\lambda_{L}$-Selmer group of $A_{L}$ is the set of all 
$c$ in $H^1(K,A[\lambda])$ such that for every $v$, $\loc_v(c)$ is in the image of the 
composition of the {\em lower} two maps.
Our methods apply directly to the Selmer groups $\Sel(L/K,A[\lambda])$, but
for our applications we are interested in the classical Selmer group.  The following 
lemma shows that these two definitions give the same Selmer groups. 
\end{rem}

\begin{lem}
\label{scl}
The isomorphism of Proposition \ref{ros}(iii) identifies $\Sel(L/K,A[\lambda])$ with the classical 
$\lambda_{L}$-Selmer group of $A_{L}$.
\end{lem}

\begin{proof}
We will show that for every place $v$, the image of the composition 
of the upper maps in \eqref{classsel} coincides with the image of the composition of the 
lower maps, and then the lemma follows from the definitions of the respective 
Selmer groups.  We will do this by constructing a vertical isomorphism on the left-hand side 
of \eqref{classsel} that makes the diagram commute.

Let $G := \Gal(L/K)$ and $G_v := \Gal(L_v/K_v)$.  The choice of place of $L$ above $v$ 
induces an isomorphism
\begin{equation}
\label{is2}
\OO[G] \otimes_{\OO[G_v]} A(L_v) \isom A(K_v \otimes_K L).
\end{equation}
Using Proposition \ref{ros}(iv) and \eqref{is2} we have
\begin{multline}
\label{is3}
A_{L}(K_v) = \I_L \otimes_{\OO[G]} A(K_v \otimes_K L) \\
   = \I_L \otimes_{\OO[G]} (\OO[G] \otimes_{\OO[G_v]} A(L_v)) 
   = \I_L \otimes_{\OO[G_v]} A(L_v).
\end{multline}

Suppose first that $L_v = K_v$, so $A_{L_v} = A$ in \eqref{classsel}.  Tensoring \eqref{is3} with 
$R_L/\lambda_L$ gives 
$$
A_{L}(K_v)/\lambda_{L}A_{L}(K_v) \cong A(K_v) \otimes_\OO \I_L/\lambda_L\I_L 
   \cong A(K_v)/\lambda A(K_v)
$$
and inserting this isomorphism into \eqref{classsel} gives a commutative diagram.
This proves the lemma in this case.

Now suppose $L_v \ne K_v$.  The inclusion 
$\OO[G_v] \hookto \OO[G]$ induces an isomorphism 
\begin{equation}
\label{is1}
\OO[G] \otimes_{\OO[G_v]} \I_{L_v} \isom \I_{L}
\end{equation}
(using here that $L_v \ne K_v$).
Using Proposition \ref{ros}(iv) (with $K_v$ in place of $K$, and $D = K_v$) and \eqref{is1} we have
\begin{multline*}
\I_L \otimes_{\OO[G_v]} A(L_v)
   = (\OO[G] \otimes_{\OO[G_v]} \I_{L_v}) \otimes_{\OO[G_v]} A(L_v) \\
   = \OO[G] \otimes_{\OO[G_v]} A_{L_v}(K_v) 
   = R_L \otimes_{R_{L_v}} A_{L_v}(K_v)
\end{multline*}
since $\OO[G_v]$ acts on $A_{L_v}$ through $R_{L_v}$.
Combining this with \eqref{is3} gives the first equality of 
\begin{multline*}
A_{L}(K_v)/\lambda_{L}A_{L}(K_v) = A_{L_v}(K_v) \otimes_{R_{L_v}} (R_{L}/\lambda_{L}) \\
   = A_{L_v}(K_v) \otimes_{R_{L_v}} (R_{L_v}/\lambda_{L_v}) 
   = A_{L_v}(K_v)/\lambda_{L_v}A_{L_v}(K_v)
\end{multline*}
and the second follows from the natural isomorphism 
$R_{L_v}/\lambda_{L_v} \cong R_L/\lambda_L$ (again using that $L_v \ne K_v$).
As in the previous case, inserting this isomorphism into \eqref{classsel} 
gives a commutative diagram and completes the proof of the lemma.
\end{proof}

\begin{prop}
\label{tower}
Suppose $L/K$ is a cyclic extension of degree $\ell^n$.  Then
$$
\rk_\Z(A(L)) \le \rk_\Z(A(K)) + \rk_\Z(\OO) \; 
   \sum_{i=1}^{n} \varphi(\ell^i)\dim_{\OO/\lambda}(\Sel(L_i/K,A[\lambda]))
$$
where $L_i$ is the extension of $K$ of degree $\ell^i$ in $L$.
\end{prop}

\begin{proof}
There is an isogeny
$$
\bigoplus_{i=0}^n A_{L_i} \too \Res^L_K A
$$
defined over $K$ (see for example \cite[Theorem 3.5]{alc} or \cite[Theorem 5.2]{mrs}).  
Since $A_{L_0} = A$, and $(\Res^{L_i}_K A)(K) = A(L_i)$, taking the $K$-points yields
\begin{equation}
\label{ri}
\rk_\Z A(L) = \rk_\Z A(K) + \sum_{i=1}^{n} \rk_\Z A_{L_i}(K).
\end{equation}

For every $i$, by Lemma \ref{scl} the Kummer map gives an injection 
$$
A_{L_i}(K) \otimes (R_{L_i}/\lambda_{L_i}) \hookto \Sel(L_i/K,A[\lambda]).
$$
For every $i$ the natural map $\OO \to R_{L_i}$ induces an isomorphism 
$\OO/\lambda \to R_{L_i}/\lambda_{L_i}$, 
and $\rk_\Z(R_{L_i}) = \varphi(\ell^i) \, \rk_\Z(\OO)$, so
\begin{align*}
\rk_\Z A_{L_i}(K) &= \rk_\Z(R_{L_i}) \; \rk_{R_{L_i}}(A_{L_i}(K)) \\
   &\le \varphi(\ell^i) \; \rk_\Z(\OO) \; \dim_{\OO/\lambda}(A_{L_i}(K) \otimes (R_{L_i}/\lambda_{L_i})) \\
   &\le \varphi(\ell^i) \; \rk_\Z(\OO) \; \dim_{\OO/\lambda}(\Sel(L_i/K,A[\lambda])).
\end{align*}
Combined with \eqref{ri} this proves the inequality of the proposition.
\end{proof}

\section{Twisting to decrease the Selmer rank}
\label{decrease}
\label{pav}

In this section we carry out the main argument of the proof of Theorem \ref{static-avs}.  
Namely, we show how to choose good local conditions on the fields $L$ so that 
the corresponding relative Selmer groups $\Sel(L/K,A[\lambda])$ vanish.

Let $A/K$, $\ell^n$, and $\lambda$ be as in the previous sections.  
Let $\EE := \End_K(A)$, and recall that $\OO$ is the center of $\EE$.   
We will abbreviate $\F_\lambda := \OO/\lambda$ and $\EE/\lambda := \EE \otimes_\OO \F_\lambda$, so in 
particular $A[\lambda]$ is an $\EE/\lambda$-module.  Fix a polarization of $A$, and 
let $\alpha \mapsto \alpha^\dagger$ denote the Rosati involution of $\EE$ corresponding 
to this polarization.

\begin{defn}
The ring $M_d(\F_\lambda)$ of $d \times d$ matrices with entries in $\F_\lambda$  
has a unique (up to isomorphism) simple left module, namely $\F_\lambda^d$ 
with the natural action.  If $R$ is any ring isomorphic to $M_d(\F_\lambda)$, 
$W$ is a simple left $R$-module, and  
$V$ is a finitely generated left $\EE/\lambda$-module, 
then $V \cong W^r$ for some $r$ and we call $r$ the {\em length} of $V$, so that
$$
\lth_{\EE/\lambda}V = \frac{1}{d} \dim_{\F_\lambda}V.
$$
\end{defn}
 
For this section we assume in addition that:
\begin{align}
\label{p5}\tag{H.1}
&\text{$\ell \ge 3$ and $\ell$ does not divide the degree of our fixed polarization,} \\
\label{p7}\tag{H.2}
&\text{there are isomorphisms $\EE \otimes_\OO \M_\lambda \cong M_d(\M_\lambda)$, 
$\EE/\lambda \cong M_d(\F_\lambda)$ for some $d$,} \\
\label{p6}\tag{H.3}
&\text{$A[\lambda]$ and $\A[\lambda^\dagger]$ are irreducible $\EE[G_K]$-modules,}\\
\label{p4}\tag{H.4}
&\text{$H^1(K(A[\lambda])/K,A[\lambda]) = 0$ and $H^1(K(A[\lambda^\dagger])/K,A[\lambda^\dagger]) = 0$,}\\
\label{p3}\tag{H.5}
&\text{there is no abelian extension of degree $\ell$ of $K(\bmu_\ell)$ in $K(\bmu_\ell,A[\lambda])$,}\\
\label{p1}\tag{H.6}
&\text{there is a $\tau_0 \in G_{K(\bmu_{\ell})}$ such that 
$A[\lambda]/(\tau_0-1)A[\lambda] = 0$,} \\
\label{p2}\tag{H.7}
&\text{there is a $\tau_1 \in G_{K(\bmu_{\ell})}$ such that 
$\lth_{\EE/\lambda}(A[\lambda]/(\tau_1-1)A[\lambda]) = 1$}.
\end{align}

We will show in \S\ref{hyps} below, using results of Serre, that almost all $\ell$ satisfy 
\eqref{p5} through \eqref{p3}.  If $K$ is sufficiently large, 
then it follows from results of Larsen in the Appendix that 
\eqref{p1} and \eqref{p2} hold for a set of primes $\ell$ of positive 
density.

Suppose $U$ is a finitely generated subgroup of $K^\times$, and 
consider the following diagram:

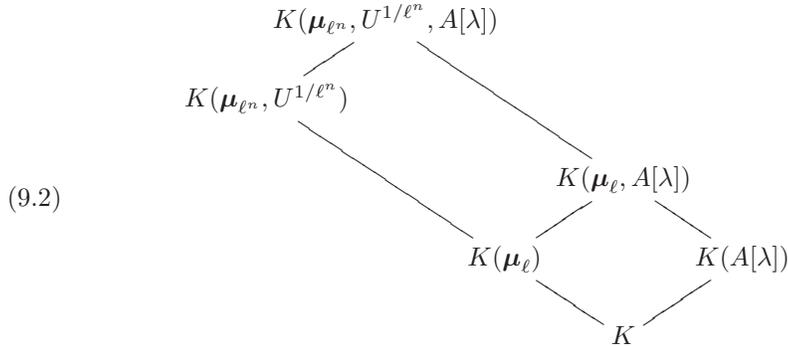
\begin{equation}
\label{Udiag}
\raisebox{70pt}{
\xymatrix@C=45pt@R=30pt@!C0{
& K(\bmu_{\ell^n},U^{1/\ell^n},A[\lambda]) \\
K(\bmu_{\ell^n},U^{1/\ell^n}) \ar@{-}[ur] \\
&&& K(\bmu_{\ell},A[\lambda]) \ar@{-}[uull]\\
&&K(\bmu_\ell) \ar@{-}[ur] \ar@{-}[uull] && K(A[\lambda]) \ar@{-}[ul]\\
&&&K \ar@{-}[ur]\ar@{-}[ul]
}}
\end{equation}

\begin{lem}
\label{6.5a}
If $U$ is a finitely generated subgroup of $K^\times$, then in the diagram \eqref{Udiag} we have
$$
K(\bmu_{\ell^n},U^{1/\ell^n}) \cap K(\bmu_{\ell},A[\lambda]) = K(\bmu_{\ell}).
$$
\end{lem}

\begin{proof}
Let $F := K(\bmu_{\ell^n},U^{1/\ell^n}) \cap K(\bmu_{\ell},A[\lambda])$.  
Then $F/K(\bmu_\ell)$ is a Galois
$\ell$-exten\-sion, so if $F \ne K(\bmu_\ell)$ then $F$ contains a cyclic extension 
$F'/K(\bmu_\ell)$ of degree $\ell$.  But since $F' \subset K(\bmu_{\ell},A[\ell])$, this is 
impossible by \eqref{p3}.  This proves the lemma.
\end{proof}

\begin{lem}
\label{6.5b}
If $U$ is a finitely generated subgroup of $K^\times$, then the restriction map 
$$
H^1(K,A[\lambda]) \too H^1(K(\bmu_{\ell^n},U^{1/\ell^n},A[\lambda]),A[\lambda])
$$
is injective.
\end{lem}

\begin{proof}
Let $F := K(\bmu_{\ell^n},U^{1/\ell^n})$. 
Restriction gives a composition
\begin{equation}
\label{sg}
\Gal(F(A[\lambda])/F) \isom \Gal(K(\bmu_\ell,A[\lambda])/K(\bmu_\ell)) 
   \hookto \Gal(K(A[\lambda])/K)
\end{equation}
where the first map is an isomorphism by Lemma \ref{6.5a}, and 
the second map is injective with cokernel of order prime to $\ell$.
The restriction map in the lemma is the composition of two restriction maps
$$
H^1(K,A[\lambda])
    \map{\;f_1\;} H^1(F, A[\lambda])
     \map{\,f_2\;} H^1(F(A[\lambda]),A[\lambda]).
$$
By \eqref{sg} and \eqref{p4}, we have
$$
\ker(f_2) = H^1(F(A[\lambda])/F,A[\lambda]) = H^1(K(A[\lambda])/K,A[\lambda]) = 0.
$$
Further,
$$
\ker{f_1} = H^1(F/K,A(F)[\lambda]).
$$
If $\tau_0 \in \Gal(K(\bmu_\ell,A[\lambda])/K(\bmu_\ell))$ is as in \eqref{p1}, 
then by \eqref{sg} we can find $\tau_0' \in \Gal(F(A[\lambda])/F)$ that 
restricts to $\tau_0$.  But then $\tau_0'$ has no nonzero fixed points in 
$A[\lambda]$.  Hence $A(F)[\lambda] = 0$, so $\ker(f_1) = 0$ 
as well and the proof is complete.
\end{proof}

\begin{lem}
\label{cocyclem}
Suppose $F$ is a Galois extension of $K$ containing $K(A[\lambda])$, and $c$ is a cocycle 
representing a class in $H^1(K,A[\lambda])$ whose restriction to $F$ is nonzero.  
If $\sigma \in G_K$ and $(\sigma-1)A[\lambda] \ne A[\lambda]$, 
then the restriction of $c$ to $G_F$ induces a nonzero homomorphism
$$
G_F \too A[\lambda]/(\sigma-1)A[\lambda].
$$
\end{lem}

\begin{proof}
Since $G_F$ acts trivially on $A[\lambda]$, the restriction of $c$ to $G_F$ is a 
(nonzero, by assumption) homomorphism $f : G_F^\ab \to A[\lambda]$.  
Recall that $\EE := \End_K(A)$, and 
let $D \subset A[\lambda]$ denote the $\EE$-module generated by the image of $f$.
Since $c$ is a lift from $K$, we have that $f$ is $G_K$-equivariant, and in particular 
$D$ is a nonzero $\EE[G_K]$-submodule of $A[\lambda]$.  By \eqref{p6} it follows that 
$D = A[\lambda]$.  But $(\sigma-1)A[\lambda]$ is a proper $\EE$-stable submodule of $A[\lambda]$, 
so the image of $f$ cannot be contained in $(\sigma-1)A[\lambda]$.
\end{proof}

Recall we have fixed a polarization of $A$ of degree prime to $\ell$ (by \eqref{p5}), 
and $\alpha \mapsto \alpha^\dagger$ is the corresponding Rosati involution of $\EE$.  The polarization induces 
a nondegenerate pairing $A[\ell] \times A[\ell] \to \bmu_\ell$, which restricts to a nondegenerate pairing
\begin{equation}
\notag 
A[\lambda] \times A[\lambda^\dagger] \to \bmu_\ell
\end{equation}
and induces an isomorphism
\begin{equation}
\label{dualiso}
A[\lambda^\dagger] \cong \Hom(A[\lambda], \bmu_\ell).
\end{equation}
Note that if conditions \eqref{p5} through \eqref{p2} hold for $\lambda$, then they also 
hold for $\lambda^\dagger$ (with the same $\tau_0$ and $\tau_1$).

\begin{defn}
If $\a$ is an ideal of $\OK$,  
define relaxed-at-$\a$ and strict-at-$\a$ Selmer groups to be, respectively,
\begin{align*}
\Sel(K,A[\lambda])^\a &:= \{c \in H^1(K,A[\lambda]) : \text{$\loc_v(c) \in \H_{\lambda}(K_v)$ for every $v \nmid \a$}\},\\
\Sel(K,A[\lambda])_\a &:= \{c \in \Sel(K,A[\lambda])^\a : \text{$\loc_v(c) = 0$ for every $v \mid \a$}\},
\end{align*}
and similarly with $\lambda$ replaced by $\lambda^\dagger$.  Note that
$$
\Sel(K,A[\lambda])_\a \subset \Sel(K,A[\lambda]) \subset \Sel(K,A[\lambda])^\a.
$$
\end{defn}

\begin{defn}
\label{Qdef}
From now on let $\Sigma$ be a finite set of places of $K$ containing
all places where $A$ has bad reduction, all places dividing $\ell\infty$,
and large enough so that the primes in $\Sigma$ generate the ideal class group of $K$.
Define
$$
\O_{K,\Sigma} := \{x \in K : \text{$x \in \O_{K_v}$ for every $v \notin \Sigma$}\},
$$
the ring of $\Sigma$-integers of $K$.
Define sets of primes $\P \subset \cQ$ by
\begin{align*}
\cQ &:= \{\p \notin \Sigma : \bN\p \equiv 1 \pmod{\ell^n}\} \\
\P &:= \{\p\in\cQ : \text{the inclusion $K^\times \hookto K_\p^\times$ 
   sends $\O_{K,\Sigma}^\times$ into $(\O_{K_\p}^\times)^{\ell^n}$}\}.
\end{align*}
Note that the action of $\EE$ on $A[\lambda]$ makes $\Hu(K_\p,A[\lambda])$ an $\EE$-module.
Define partitions of $\P, \cQ$ into disjoint subsets $\P_i, \cQ_i$ for $i \ge 0$ by
$$
\cQ_i := \{\p\in\cQ : \lth_{\EE/\lambda}\Hu(K_\p,A[\lambda]) = i\}, \quad \P_i := \cQ_i \cap \P
$$
and if $\a$ is an ideal of $\O_K$, let $\P_1(\a)$ be the subset of all $\p \in \P_1$ such that 
the localization maps
$$
\Sel(K,A[\lambda])_{\a} \map{\loc_\p} \Hu(K_\p,A[\lambda]), \quad 
   \Sel(K,A[\lambda^\dagger])_{\a} \map{\loc_\p} \Hu(K_\p,A[{\lambda^\dagger}])
$$
are both nonzero.  

Note that by Lemma \ref{urrem}(i) and \eqref{dualiso}, if $\p\in\cQ_i$ 
then $\lth_{\EE/\lambda^\dagger}\Hu(K_\p,A[\lambda^\dagger]) = i$ as well.
\end{defn}

In the language of the Introduction and \S\ref{intro2}, the {\em critical primes} 
are the primes in $\cQ_1$ and the {\em silent primes} are the primes in $\cQ_0$.

\begin{prop}
\label{goodp}
\begin{enumerate}
\item
The sets $\P_0$ and $\P_1$ have positive density.
\item
Suppose $\a$ is an ideal of $\OK$ such that both $\Sel(K,A[\lambda])_\a$ and $\Sel(K,A[\lambda^\dagger])_\a$ 
are nonzero.  Then $\P_1(\a)$ has positive density, and if $\p\in\P_1(\a)$ then 
\begin{align*}
\lth_{\EE/\lambda}\Sel(K,A[\lambda])_{\a\p} &= \lth_{\EE/\lambda}\Sel(K,A[\lambda])_{\a} - 1, \\
\lth_{\EE/\lambda^\dagger}\Sel(K,A[\lambda^\dagger])_{\a\p} 
   &= \lth_{\EE/\lambda^\dagger}\Sel(K,A[\lambda^\dagger])_{\a} - 1.
\end{align*}
\end{enumerate}
\end{prop}

\begin{proof}
Let $\tau_0, \tau_1$ be as in \eqref{p1} and \eqref{p2}.  By Lemma \ref{6.5a}, 
$$
K(\bmu_\ell,A[\lambda]) \cap K(\bmu_{\ell^n},(\O_{K,\Sigma}^\times)^{1/\ell^n}) = K(\bmu_\ell),
$$
so for $i = 0$ or $1$ we can choose $\sigma_i \in G_K$ such that 
\begin{align}
\label{si1}
&\text{$\sigma_i = \tau_i$ on $A[\lambda]$},\\
\label{si2}
&\text{$\sigma_i = 1$ on $K(\bmu_{\ell^n},(\O_{K,\Sigma}^\times)^{1/\ell^n})$}.
\end{align}
Fix $i = 0$ or $1$, and suppose that $\p$ is a prime of $K$ whose Frobenius 
conjugacy class in 
$\Gal(K(\bmu_{\ell^n},(\O_{K,\Sigma}^\times)^{1/\ell^n},A[\lambda])/K)$ is the class 
of $\sigma_i$.  
Since Frobenius fixes $\bmu_{\ell^n}$ and $(\O_{K,\Sigma}^\times)^{1/\ell^n}$ by 
\eqref{si2}, we have that $\bmu_{\ell^n}$ and $(\O_{K,\Sigma}^\times)^{1/\ell^n}$ 
are contained in $K_\p^\times$.  Hence $\bN\p \equiv 1 \pmod{\ell^n}$ and 
the inclusion $K^\times \hookto K_\p^\times$ sends $\O_{K,\Sigma}^\times$ into 
$(\O_{K,\p}^\times)^{\ell^n}$, so by definition $\p\in\P$.

By \eqref{si1} and Lemma \ref{urrem}, evaluation of cocycles on a Frobenius 
element for $\p$ in $G_K$ induces an isomorphism
\begin{equation}
\label{last}
\H_{\lambda}(K_\p) = \Hu(K_\p,A[\lambda]) \map{~\sim~} A[\lambda]/(\tau_i-1)A[\lambda]
\end{equation}
and similarly for $\lambda^\dagger$. 
Thus $\p\in\P_i$, so the Cebotarev Theorem shows that $\P_0$ and $\P_1$ have positive density.  
This is (i).

Fix an ideal $\a$ of $\O_K$ and suppose that $c$ and $d$ are cocycles representing 
nonzero elements of $\Sel(K,A[\lambda])_\a$ and 
$\Sel(K,A[\lambda^\dagger])_\a$, respectively.  Let 
$$
F := K(\bmu_{\ell^n},(\O_{K,\Sigma}^\times)^{1/\ell^n},A[\lambda]),
$$
and let $\sigma_1$ be as above.
By Lemmas \ref{6.5b} and \ref{cocyclem}, the restrictions of $c$ and $d$ to $G_F$ 
induce nonzero homomorphisms 
$$
\tilde{c} : G_F \too A[\lambda]/(\sigma_1-1)A[\lambda], \quad  
\tilde{d} : G_F \too A[\lambda^\dagger]/(\sigma_1-1)A[\lambda^\dagger].  
$$
Let $Z_c$ be the subset of all $\gamma \in G_F$ such that 
$c(\gamma) = -c(\sigma_1)$ in $A[\lambda]/(\sigma_1-1)A[\lambda]$, and similarly 
for $Z_d$ with $\lambda$ replaced by $\lambda^\dagger$.  Since $\tilde{c}$ and $\tilde{d}$ 
are nonzero, $Z_c$ and $Z_d$ 
each have Haar measure at most $1/\ell$ in $G_F$, so $Z_c \cup Z_d \ne G_F$ 
(this is where we use that $\ell \ge 3$ in assumption \eqref{p5}).

Thus we can find $\gamma \in G_F$ such that $\tilde{c}(\gamma\sigma_1) \ne 0$ 
and $\tilde{d}(\gamma\sigma_1) \ne 0$.  Since $\gamma$ acts trivially on $A[\lambda]$, 
this means that
$$
c(\gamma\sigma_1) \notin (\sigma_1-1)A[\lambda] = (\gamma\sigma_1-1)A[\lambda]  
$$
and similarly for $d$.
Let $N$ be a Galois extension of $K$ containing $F$ and such that the restrictions 
of $c$ and $d$ to $G_F$ factor through $\Gal(N/F)$.
If $\p$ is a prime whose Frobenius conjugacy class in $\Gal(N/K)$ is the class of 
$\gamma\sigma_1$, then $\loc_\p(c) \ne 0$ and $\loc_\p(d) \ne 0$, so 
$\p\in\P_1(\a)$.  Now the Cebotarev Theorem shows that $\P_1(\a)$ has positive density.

If $\p\in\P_1(\a)$ then we have exact sequences of $\EE/\lambda$ and $\EE/\lambda^\dagger$-modules
\begin{gather*}
0 \too \Sel(K,A[\lambda])_{\a\p} \too \Sel(K,A[\lambda])_{\a} \map{\loc_\p} \Hu(K_\p,A[{\lambda}]) \too 0 \\
0 \too \Sel(K,A[\lambda^\dagger])_{\a\p} \too \Sel(K,A[\lambda^\dagger])_{\a} 
   \map{\loc_\p} \Hu(K_\p,A[{\lambda^\dagger}]) \too 0 
\end{gather*}
where the right-hand maps are surjective because they are nonzero and (by \eqref{last}) 
the target modules are simple.  This completes the proof of (ii).
\end{proof}

\begin{defn}
\label{updef}
Suppose $T$ is a finite set of primes of $K$, disjoint from $\Sigma$.  
We will say that an extension $L/K$ is 
{\em $T$-ramified and $\Sigma$-split} if every $\p\in T - \cQ_0$ is totally ramified in $L/K$, 
every $\p \notin T$ is unramified in $L/K$, and every $v \in \Sigma$ splits completely in $L/K$.
\end{defn}

The primes in $\cQ_0$ are the silent primes referred to in the Introduction and \S\ref{intro2}.  
The local Selmer conditions at these primes are zero, so we need no condition on their splitting 
behavior in Definition \ref{updef}.

\begin{lem}
\label{l7.15}
Suppose $T$ is a nonempty finite subset of $\P$,  
and let $T_0 := T \cap \P_0$.  For each $\p \in T_0$ fix $e_\p$ with $0 \le e_\p \le n$.  
If $T = T_0$ assume in addition that some $e_\p = n$.
Then there is a cyclic extension $L/K$ of degree $\ell^n$ that is $T$-ramified and $\Sigma$-split, 
and such that if $\p \in T_0$ then the ramification degree of $\p$ in $L/K$ is $\ell^{e_\p}$.
\end{lem}

\begin{proof}
Suppose $\p\in\P$.
Let $\bA_K^\times$ denote the group of ideles of $K$, and let 
$K(\p)$ be the abelian extension of $K$ corresponding by global class field theory 
to the subgroup
$$
Y := K^\times (\O_{K_\p}^\times)^{\ell^n} 
   \prod_{v\in\Sigma}K_v^\times\prod_{v\notin\Sigma\cup\{\p\}}\O_{K_v}^\times
   \subset \bA_K^\times.
$$
Class field theory tells us that the inertia (resp., decomposition) group of a place $v$ 
in $\Gal(K(\p)/K)$ is the image of $\O_{K_v}^\times$ (resp., $K_v^\times$) in $\bA_K^\times/Y$.
If $v \nmid p$ then $\O_{K_v}^\times \subset Y$, so
$K(\p)/K$ is unramified outside of $\p$.  If $v \in \Sigma$ then $K_v^\times \subset Y$, 
so every $v \in \Sigma$ splits 
completely in $K(\p)/K$.  Since $\Sigma$ was chosen large enough to generate the 
ideal class group of $K$, the natural map $\O_{K_\p}^\times \to \bA^\times_K/Y$ is surjective, 
so $K(\p)/K$ is totally ramified at $\p$.  It follows from the definition of $\P$
that $\Gal(K(\p)/K) \cong \bA_K^\times/Y$ is cyclic of order $\ell^n$.  Now we can find 
an extension that is $T$-ramified and $\Sigma$-split, with the desired ramification 
degree at primes in $T_0$, inside the compositum of the 
fields $K(\p)$ for $\p\in T$.
\end{proof}

\begin{lem}
\label{sames}
Suppose $T$ is a finite subset of $\P$, and $L/K$ is a cyclic extension of degree 
$\ell^n$ that is $T$-ramified and $\Sigma$-split.  If $K \subsetneq L' \subset L$ 
then $\Sel(L'/K,A[\lambda]) =\Sel(L/K,A[\lambda])$.
\end{lem}

\begin{proof}
We will show that $\H_\lambda(L'_v/K_v) = \H_\lambda(L_v/K_v)$ for every $v$.  
If $v \in \Sigma$ this holds because $L_v' = L_v = K_v$.  If $v \in T - \P_0$ 
this holds by Proposition \ref{gl}(i). 
If $v \notin \Sigma \cup T$ this holds by Lemma \ref{urrem}(ii).  
Finally, if 
$v \in \P_0$ then $\H_\lambda(L'_v/K_v) = \H_\lambda(L_v/K_v) = 0$ by 
Lemmas \ref{xlem}(ii) and \ref{urrem}(i). 
Thus the two Selmer groups coincide in $H^1(K,A[\lambda])$.
\end{proof}

In the terminology of the Introduction and \S\ref{intro2}, we next 
use critical primes (those in $\P_1$) 
to decrease the rank of the Selmer group, while the silent primes 
(those in $\P_0$) have no effect on the rank.

\begin{prop}
\label{l7.13}
Let $r := \lth_{\EE/\lambda}\Sel(K,A[\lambda])$, 
$r^\dagger := \lth_{\EE/\lambda^\dagger}\Sel(K,A[\lambda^\dagger])$,
and suppose that $t \le \min\{r,r^\dagger\}$.
\begin{enumerate}
\item
There is a set of primes $\TT \subset \P_1$ of cardinality $t$ 
such that 
$$
\lth_{\EE/\lambda}\Sel(K,A[\lambda])_\a = r-t, \quad 
   \lth_{\EE/\lambda^\dagger}\Sel(K,A[\lambda^\dagger])_\a = r^\dagger-t,
$$
where $\a := {\prod_{\p\in T}\p}$.
\item
If $\TT$ is as in (i), $T_0$ is a finite subset of $\cQ_0$, 
and $L/K$ is a cyclic extension of $K$ 
of degree $\ell^n$ that is $(T_0 \cup \TT)$-ramified and $\Sigma$-split, then 
$$
\lth_{\EE/\lambda}\Sel(L/K,A[\lambda]) = r-t,
   \quad \lth_{\EE/\lambda^\dagger}\Sel(L/K,A[\lambda^\dagger]) = r^\dagger-t.
$$
\end{enumerate}
\end{prop}

\begin{proof}
We will prove (i) by induction on $t$.  When $t = 0$ there is nothing to check.

Suppose $\TT$ satisfies the conclusion of the lemma for $t$, and $t < \min\{r,r^\dagger\}$. 
Let $\a := {\prod_{\p\in \TT}\p}$.  
Then we can apply Proposition \ref{goodp}(ii), to choose $\p \in \P_1(\a)$ 
so that 
$$
\lth_{\EE/\lambda}\Sel(K,A[\lambda])_{\a\p} = r-t-1, \quad 
   \lth_{\EE/\lambda^\dagger}\Sel(K,A[\lambda^\dagger])_{\a\p} = r^\dagger-t-1.
$$  
Then $\TT \cup \{\p\}$ satisfies the conclusion of (i) for $t+1$.

Now suppose that $\TT$ is such a set, and $\a := {\prod_{\p\in \TT}\p}$.  Consider the exact sequences
\begin{equation}
\label{gdd}
\raisebox{19pt}{
\xymatrix@C=12pt@R=7pt{
0 \ar[r] & \Sel(K,A[\lambda]) \ar[r] & \Sel(K,A[\lambda])^{\a} \ar^-{\oplus \loc_\p}[rr] 
   && \dirsum{\p\in \TT}H^1(K_\p,A[\lambda])/\H_{\lambda}(K_\p) \\
0 \ar[r] & \Sel(K,A[\lambda^\dagger])_{\a} \ar[r] & \Sel(K,A[\lambda^\dagger]) \ar^{\oplus \loc_\p}[rr] 
   && \dirsum{\p\in \TT}\H_{\lambda^\dagger}(K_\p).
}}
\end{equation}
Using \eqref{dualiso} to identify $A[\lambda^\dagger]$ with the dual of $A[\lambda]$,
the local conditions that define the Selmer groups $\Sel(K,A[\lambda])$ and $\Sel(K,A[\lambda^\dagger])$ 
(resp.\ $\Sel(K,A[\lambda])^{\a}$ and $\Sel(K,A[\lambda^\dagger])_{\a}$)
are dual Selmer structures in the sense of \cite[\S2.3]{kolysys}.
Thus we can use global duality (see for example \cite[Theorem 2.3.4]{kolysys}) to
conclude that the images of the two right-hand maps in \eqref{gdd} are orthogonal complements of each other 
under the sum of the local Tate pairings.  By our choice of $\TT$ the lower right-hand map 
is surjective, so the upper right-hand map is zero, i.e., 
\begin{equation}
\label{siss}
(\oplus_{\p\in \TT} \loc_\p)(\Sel(K,A[\lambda])^{\a})  \subset \dirsum{\p\in \TT}\H_\lambda(K_\p).
\end{equation}
Let $T_0$ be a finite subset of $\cQ_0$, let $\b := \prod_{\p\in T_0}\p$, and suppose 
$L$ is a cyclic extension that is $(T_0 \cup \TT)$-ramified and $\Sigma$-split.
By definition (and Lemma \ref{urrem}(ii)), $\Sel(L/K,A[\lambda])$ is the kernel of the map
$$
\Sel(K,A[\lambda])^{\a\b} \map{\oplus_{\p\in T_0 \cup \TT} \loc_\p} 
   \dirsum{\p\in T_0 \cup \TT}H^1(K_\p,A[\lambda])/\H_\lambda(L_\p/K_\p).
$$
We have $\H_\lambda(K_\p) = \H_\lambda(L_\p/K_\p) = 0$ for every $\p\in\cQ_0$ 
by Lemmas \ref{xlem}(ii) and \ref{urrem}(i) and the 
definition of $\cQ_0$, so in fact $\Sel(L/K,A[\lambda])$ is the kernel of the map
\begin{equation}
\label{siss2}
\Sel(K,A[\lambda])^\a \map{\oplus_{\p\in \TT} \loc_\p} \dirsum{\p\in \TT}H^1(K_\p,A[\lambda])/\H_\lambda(L_\p/K_\p).
\end{equation}
By Proposition \ref{gl}(ii), $\H_\lambda(K_\p) \cap \H_\lambda(L_\p/K_\p) = 0$ for every $\p \in \P_1$.  
Combining \eqref{siss} and \eqref{siss2} shows that $\Sel(L/K,A[\lambda]) = \Sel(L/K,A[\lambda])_\a$, 
so by our choice of $\TT$ we have $\lth_{\EE/\lambda}\Sel(L/K,A[\lambda]) = r-t$.
The proof for $\lambda^\dagger$ is the same.
\end{proof}

\begin{thm}
\label{ranks}
Suppose that \eqref{p5} through \eqref{p2} all hold, 
and $n \ge 1$. 
Then for every finite set $\Sigma$ of primes of $K$, 
there are infinitely many cyclic extensions $L/K$ of degree $\ell^n$, 
completely split at all places in $\Sigma$, such that $A(L) = A(K)$.
\end{thm}

\begin{proof}
Enlarge $\Sigma$ if necessary so that the conditions of 
Definition \ref{Qdef} are satisfied.
We may also assume without loss of generality that 
$$
\lth_{\EE/\lambda}\Sel(K,A[\lambda]) \le \lth_{\EE/\lambda^\dagger}\Sel(K,A[\lambda^\dagger])
$$
(if not, we can simply switch $\lambda$ and $\lambda^\dagger$; all the properties 
we require for $\lambda$ hold equivalently for $\lambda^\dagger$, using the isomorphism 
\eqref{dualiso}). 
Apply Proposition \ref{l7.13}(i) with $t := \lth_{\EE/\lambda}\Sel(K,A[\lambda])$ 
to produce a finite set $\TT \subset \P_1$.

Now suppose that $T_0$ is a finite subset of $\cQ_0$.  If $L/K$ is cyclic of degree $\ell^n$, 
$(T_0 \cup \TT)$-ramified and $\Sigma$-split, then Proposition \ref{l7.13} shows  
$\Sel(L/K,A[\lambda]) = 0$.  Further, Lemma \ref{sames} shows that $\Sel(L'/K,A[\lambda]) = 0$
if $K \subsetneq L' \subset L$, so by Proposition \ref{tower} we have $\rk(A(L)) = \rk(A(K))$.

Since $\P_0$ has positive density (Proposition \ref{goodp}(i)), 
there are infinitely many finite subsets $T_0$ of $\P_0 \subset \cQ_0$. 
For each such $T_0$, Lemma \ref{l7.15} shows that there is a cyclic extension $L/K$ 
of degree $\ell^n$ that is $(T_0 \cup \TT)$-ramified and $\Sigma$-split, and totally ramified at all 
primes in $T_0$ as well.  These fields are all distinct, so we have 
infinitely many different $L$ with $\rk(A(L)) = \rk(A(K))$.

Now suppose that the set $T_0$ in the construction above contains primes $\p_1, \p_2$ 
with different residue characteristics.  In particular $L/K$ is totally ramified 
at $\p_1$ and $\p_2$.  If $A(L) \ne A(K)$, then (since $\rk(A(L)) = \rk(A(K))$) 
there is a prime $p$ and point $x \in A(L)$ such that $x \notin A(K)$ but $px \in A(K)$.
It follows that the extension $K(x)/K$ is unramified outside of $\Sigma$ and primes above $p$.  
In particular $K \subset K(x) \subset L$ but 
$K(x)/K$ cannot ramify at both $\p_1$ and $\p_2$, so we must have $K(x) = K$, i.e., $x \in A(K)$.
This contradiction shows that $A(L) = A(K)$ for all such $T_0$, and this proves the theorem.
\end{proof}

\section{Proof of Theorem \ref{static-avs}}
\label{hyps}

\begin{prop}
\label{class}
Conditions \eqref{p5} through \eqref{p3} hold for all sufficiently large $\ell$.
\end{prop}

\begin{proof}
This is clear for \eqref{p5}.

Recall that $\lambda$ was chosen not to divide the discriminant of $\OO$, so 
$\OO_\lambda$ is the ring of integers of $\M_\lambda$.
Since $A$ is simple, $\EE \otimes \Q$ is a central simple division algebra over $\M$, 
of some degree $d$.  By the general theory of such algebras (see for example 
\cite[Proposition in \S18.5]{pierce}), for all but finitely many primes $\lambda$ of $\M$
we have 
$$\EE \otimes_\OO \M_\lambda \cong M_d(\M_\lambda).$$
If in addition $\lambda$ does not divide the index of $\EE$ in a fixed maximal order 
of $\EE \otimes_\OO \M$, then 
$$
\text{$\EE \otimes_\OO \OO_\lambda$ is a maximal order in $\EE \otimes_\OO \M_\lambda$.}
$$
By \cite[Proposition 3.5]{auslander}, every maximal order in $M_d(\M_\lambda)$ is conjugate to 
$M_d(\OO_\lambda)$, so for such $\lambda$ we have
$$
\EE/\lambda := \EE \otimes_\OO \F_\lambda 
   \cong M_d(\OO_\lambda) \otimes_\OO \F_\lambda = M_d(\F_\lambda)
$$
which is \eqref{p7}.

Condition \eqref{p6} holds for large $\ell$ by Corollary \ref{irredcor} of the Appendix. 

Let $B \subset \Gal(K(A[\lambda])/K)$ denote the subgroup acting as scalars on $A[\lambda]$.
Then $B$ is a normal subgroup and we have the inflation-restriction exact sequence
\begin{equation}
\label{irs}
H^1(K(A[\lambda])^B/K,A[\lambda]^B) \too H^1(K(A[\lambda])/K,A[\lambda]) 
   \too H^1(B,A[\lambda]).
\end{equation}
Since $B$ has order prime to $\ell$, $H^1(B,A[\lambda] = 0$.
Serre \cite[Th\'eor\`eme of \S 5]{Vigneras_na} shows that 
$B$ is nontrivial for all sufficiently large $\ell$.  When $B$ is nontrivial, 
$A[\lambda]^B = 0$ so the left-hand term in \eqref{irs} vanishes and \eqref{p4} holds.

Let $\Gamma$ denote the image of $\Gal(K(\bmu_\ell,A[\lambda])/K(\bmu_\ell))$ in 
$\Aut(A[\lambda])$. 
Then \cite[Theorem 0.2]{LPfsag} shows that there are normal subgroups 
$\Gamma_3 \subset \Gamma_2 \subset \Gamma_1$ of $\Gamma$ such that 
$\Gamma_3$ is an $\ell$-group, $\Gamma_2/\Gamma_3$ has order prime to $\ell$, 
$\Gamma_1/\Gamma_2$ is a direct product of finite simple 
groups of Lie type in characteristic $\ell$, and 
$[\Gamma:\Gamma_1]$ is bounded independently of $\ell$.
By Faltings' theorem (see for example the proof of \eqref{p6} referenced above)
$\Gamma$ acts semisimply on $A[\lambda]$ for sufficiently large $\ell$, 
and then $\Gamma_3$ must be trivial. 
It follows that if $\ell$ is sufficiently large then 
$\Gamma$ has no cyclic quotient of order $\ell$, i.e., \eqref{p3} holds.
\end{proof}

\begin{thm}[Larsen]
\label{class3}
Suppose that all $\Kb$-endomorphisms of $A$ are defined over $K$.
Then the conditions \eqref{p1} and \eqref{p2} hold simultaneously for a set of primes $\ell$
of positive density.
\end{thm}

\begin{proof}
This is Theorem \ref{main} of the Appendix.
\end{proof}

\begin{proof}[Proof of Theorem \ref{static-avs}]
If all $\Kb$-endomorphisms of $A$ are defined over $K$,
then by Proposition \ref{class} and Theorem \ref{class3} 
there is a set $S$ of rational primes with positive density such that 
our hypotheses \eqref{p5} through \eqref{p2} hold simultaneously for  
all $\ell \in S$.  Thus Theorem \ref{static-avs} follows from Theorem \ref{ranks}.
\end{proof}

\begin{proof}[Proof of Theorem \ref{static-curves}] 
Lemma \ref{lemimp} showed that Theorem \ref{static-curves} follows from Theorem \ref{static-avs}.
\end{proof}

\begin{rem}
\label{nonsimple}
It is natural to try to strengthen Theorem \ref{static-avs} by removing 
the assumption that $A$ is simple.  This generalization can be reduced to the problem, 
given a finite collection of abelian varieties, of finding many 
cyclic extensions for which they are all simultaneously diophantine-stable.

Precisely, suppose that $A_1,\ldots, A_m$ are pairwise non-isogenous absolutely simple 
abelian varieties, $\ell$ is a rational prime, and $\lambda_i$ is a prime ideal of the 
center of $\End(A_i)$ above $\ell$ for each $i$.  Suppose $\ell$ is large enough 
so that \eqref{p5} through \eqref{p3} hold for every $A_i$.

If the results of the Appendix could be extended to show that 
for every $j$ there is an element $\tau_j \in G_{K(\bmu_\ell)}$ such that 
$$
\text{$A_i[\lambda_i]/(\tau_j-1)A_i[\lambda_i]$ \;is\;}
\begin{cases}
\text{zero if $i \ne j$}, \\
\text{a nonzero simple $\End(A_j)/\lambda_j$-module if $i = j$,} 
\end{cases}
$$
then the methods of \S\ref{decrease} above would show that there is a set $S$ of 
rational primes with positive density such that for every $\ell \in S$ and 
every $n \ge 1$ there are infinitely 
many cyclic extensions $L/K$ of degree $\ell^n$ such that every $A_i$
is diophantine-stable for $L/K$.
Using the argument at the end of the proof of Theorem \ref{ranks} 
it would follow that $S$ can be chosen so that 
the same result holds for every abelian variety isogenous over $K$ to $\prod_i A_i^{d_i}$.
\end{rem}

\section{Quantitative results}
\label{quant}

Fix a simple abelian variety $A/K$ such that $\End_K(A) = \End_{\bar K}(A)$, 
and an $\ell$ such that our hypotheses \eqref{p5} through \eqref{p2} all hold.
The proof of Theorem \ref{static-avs}, and more precisely Theorem \ref{ranks}, makes 
it possible to quantify how many cyclic $\ell^n$-extensions $L/K$ are being found 
with $A(L) = A(K)$.  For simplicity we will take $n = 1$, and count cyclic $\ell$-extensions.  
Keep the notation of the previous sections.

For real numbers $X > 0$, define 
\begin{align*}
\cF_K(X) &:= \{\text{cyclic extensions $L/K$ of degree $\ell$} : \bN\d_{L/K} < X\}, \\
\cF^0_K(X) &:= \{L \in \cF_K(X) : A(L) = A(K)\},
\end{align*}
where $\bN\d_{L/K}$ denotes the absolute norm of the relative discriminant of $L/K$.  
For $\p \notin \Sigma$ let $\Frob_\p \in G_K$ denote a Frobenius automorphism for $\p$.
It follows from Definition \ref{Qdef} and Lemma \ref{urrem}(i) that 
$$
\cQ_0 := \{\p \notin \Sigma : \text{$\Frob_\p = 1$ on $\bmu_\ell$ and $\Frob_\p$ has no nonzero fixed 
   points in $A[\lambda]$}\},
$$
and let 
$$
\delta := \frac{|\{\sigma \in \Gal(K(\bmu_\ell,A[\lambda])/K(\bmu_\ell)) : \text{$\sigma$ has no nonzero fixed 
   points in $A[\lambda]\}|$}}
   {[K(\bmu_\ell,A[\lambda]):K(\bmu_\ell)]}
$$
The proof of Proposition \ref{goodp}(i) shows that $\cQ_0$ has density $\delta/[K(\bmu_\ell):K]$, 
and \eqref{p1} and \eqref{p2} show that $0 < \delta < 1$.

\begin{thm}[Wright \cite{wright}]  There is a positive constant $C$ such that 
$$
|\cF_K(X)| \sim C X^{1/(\ell-1)} \log(X)^{(\ell-1)/[K(\bmu_\ell):K]-1}
$$
as $X \to \infty$.
\end{thm}

The main result of this section is the following.

\begin{thm}
\label{quantthm}
As $X \to \infty$ we have
$$
|\cF^0_K(X)| \gg X^{1/(\ell-1)}\log(X)^{(\ell-1)\delta/[K(\bmu_\ell):K]-1}.
$$
\end{thm}

\begin{exa}
Suppose $E$ is a non-CM elliptic curve, and $\ell$ is large enough so that 
the Galois representation $G_K \to \Aut(E[\ell]) = \GL_2(\Z/\ell\Z)$ is surjective.  
Then $[K(\bmu_\ell):K] = \ell-1$, and an elementary calculation shows that the number 
of elements of $\SL_2(\Z/\ell\Z)$ with nonzero fixed points is $\ell^2$.  Thus
$\delta = 1 - \ell/(\ell^2-1)$ so in this case
$$
|\cF_K(X)| \sim C X^{1/(\ell-1)}, \quad |\cF^0_K(X)| \gg X^{1/(\ell-1)}/\log(X)^{\ell/(\ell^2-1)}.
$$
\end{exa}

The rest of this section is devoted to a proof of Theorem \ref{quantthm}.

\begin{lem}
\label{lem10.3}
There is a finite subset $T_1 \subset \cQ_0$ such that the natural map
$$
\O_{K,\Sigma}^\times/(\O_{K,\Sigma}^\times)^\ell 
   \too \prod_{v \in T_1}\O_{K_v}^\times/(\O_{K_v}^\times)^\ell
$$
is injective.
\end{lem}

\begin{proof}
Suppose $u \in \O_{K,\Sigma}^\times$ and $u \notin (K^\times)^\ell$.
Then $u \notin (K(\bmu_\ell)^\times)^\ell$, so by Lemma \ref{6.5a} 
and \eqref{p1} we can choose $\sigma \in G_K$ such that $\sigma = 1$ on $\bmu_\ell$, 
$\sigma$ has no nonzero fixed points in $A[\lambda]$, and $\sigma$ does not fix $u^{1/\ell}$. 
If $v \notin\Sigma$ and the Frobenius of $v$ on $K(\bmu_\ell,A[\lambda],(\O_{K,\Sigma}^\times)^{1/\ell})$ 
is in the conjugacy class of $\sigma$, then $v \in \cQ_0$ and $u \notin (\O_{K_v}^\times)^\ell$.
Taking a collection of such $v$ as $u$ varies gives a suitable set $T_1$.
\end{proof}

Recall that $\bA_K^\times$ denote the ideles of $K$.  Fix a set $T_1$ as in 
Lemma \ref{lem10.3}.

\begin{lem}
\label{lem10.4}
The natural composition
$$
\Hom(G_K,\bmu_\ell) \too \Hom(\bA_K^\times,\bmu_\ell) \too 
   \prod_{v \in \Sigma}\Hom(K_v^\times,\bmu_\ell)
   \prod_{v \notin\Sigma\cup T_1} \Hom(\O_{K_v}^\times,\bmu_\ell)
$$
is surjective.
\end{lem}

\begin{proof}
By class field theory and our assumption that the primes in $\Sigma$ 
generate the ideal class group of $K$, we have an isomorphism
$$ 
\Hom(G_K,\bmu_\ell) \cong 
   \Hom\biggl(\bigl(\prod_{v \in \Sigma}K_v^\times
      \prod_{v \in T_1} \O_{K_v}^\times
      \prod_{v \notin\Sigma\cup T_1} \O_{K_v}^\times\bigr)/\O_{K,\Sigma}^\times ,\bmu_\ell\biggr)
$$
Now the lemma follows by a simple argument using Lemma \ref{lem10.3}; see for example 
\cite[Lemma 6.6(ii)]{KMR}.
\end{proof}

As in the proof of Theorem \ref{ranks}, we can use Proposition \ref{l7.13} to 
fix a finite set $\TT \subset \P_1$ such that for every finite set $T_0 \subset \cQ_0$, 
and every cyclic $\ell$-extension $L/K$ that is 
\begin{itemize}
\item
$(T_0 \cup \TT$)-ramified and $\Sigma$-split, 
\item
ramified at two primes in $T_0$ of different residue characteristics,
\end{itemize}
we have $A(L) = A(K)$.

\begin{defn}
Fix two primes $\p_1,\p_2 \in \P_0 - T_1$ of different residue characteristics, and let 
$\TT' := \TT \cup \{\p_1,\p_2\}$.  For every finite subset $T_0$ of $\cQ_0 - T_1$, let 
$\cC(T_0) \subset \Hom(G_K,\bmu_\ell)$ be the subset of characters $\chi$ satisfying, 
under the class field theory surjection of Lemma \ref{lem10.4}, 
\begin{itemize}
\item
$\chi|_{K_v^\times} = 1$ if $v \in \Sigma$,
\item
$\chi|_{\O_{K_v}^\times} \ne 1$  if $v \in \TT' \cup T_0$,
\item
$\chi|_{\O_{K_v}^\times} = 1$ if $v \notin \Sigma \cup \TT' \cup T_0 \cup T_1$.
\end{itemize}
\end{defn}

\begin{lem}
\label{lem10.6}
Let $\alpha$ be the (surjective) composition of maps in Lemma \ref{lem10.4}.
Then for every finite subset 
$T_0 \subset \cQ_0-T_1$ we have $|\cC(T_0)| = |\ker(\alpha)| (\ell-1)^{|\TT'|}(\ell-1)^{|T_0|}$.
\end{lem}

\begin{proof}
This is clear from the surjectivity of $\alpha$.
\end{proof}

\begin{lem}
\label{lem10.7}
Suppose $T_0$ is a finite subset of $\cQ_0-T_1$, and $\chi \in \cC(T_0)$.  
Let $L$ be the fixed field of the kernel of $\chi$.  Then:
\begin{enumerate}
\item
$A(L) = A(K)$, 
\item
the discriminant of $L/K$ is $\prod_{\p \in \TT' \cup T_0}\p^{\ell-1}$.
\end{enumerate}
\end{lem}

\begin{proof}
The first assertion follows from the definition of $T$ above.  
For the second, by definition of $\cC(T_0)$ we have that $L/K$ is cyclic of degree $\ell$, 
totally tamely ramified at $\p \in \TT' \cup T_0$ and unramified elsewhere.
\end{proof}

\begin{proof}[Proof of Theorem \ref{quantthm}]
Define a function $f$ on ideals of $K$ by 
$$
f(\a) := 
\begin{cases}
(\ell-1)^{|T_0|} & \text {if $T_0$ is a finite subset of $\cQ_0-T_1$ and $\a = \prod_{\p\in T_0}\p$}, \\
0 & \text{if $\a$ is not a squarefree product of primes in $\cQ_0-T_1$}.
\end{cases}
$$
Then $\sum_\a f(\a)\bN\a^{-s} = \prod_{\p\in\cQ_0-T_1}(1+(\ell-1)\bN\p^{-s})$, so
$$
\log\biggl(\sum_\a f(\a)\bN\a^{-s}\biggr) \approx (\ell-1)\sum_{\p\in\cQ_0-T_1}\bN\p^{-s} 
   \approx \frac{(\ell-1)\delta}{[K(\bmu_\ell):K]}\frac{1}{\log(s-1)}
$$
where ``$\approx$'' means that the two sides are holomorphic on $\Re(s) > 1$ and 
their difference approaches a finite limit as $\Re(s)\to 1^+$.
Therefore
by a variant of the Ikehara Tauberian Theorem (see for example \cite[p.\ 322]{wintner}) 
we conclude that there is a constant $D$ such that
$$
\sum_{\bN\a < X} f(\a) \sim DX \log(X)^{(\ell-1)\delta/[K(\bmu_\ell):K] - 1}.
$$
By Lemmas \ref{lem10.6} and \ref{lem10.7}, for every $\a$ the number of cyclic $\ell$-extensions $L/K$ 
of discriminant $(\a\prod_{\p\in T'}\p)^{\ell-1}$ with $A(L) = A(K)$ is at least 
$f(\a)$, and the theorem follows.
\end{proof}

\bigskip
{\small\noindent
{\sc Department of Mathematics, Harvard University, Cambridge, MA 02138, USA}\\
{\em E-mail address:} {\href{mailto:mazur@math.harvard.edu}{\tt mazur@math.harvard.edu}}}

\bigskip
{\small\noindent
{\sc Department of Mathematics, UC Irvine, Irvine, CA 92697, USA}\\
{\em E-mail address:} {\href{mailto:krubin@uci.edu}{\tt krubin@uci.edu}}}

\newpage
\part{Appendix by Michael Larsen: Galois elements acting on $\ell$-torsion points of abelian varieties}

\renewcommand{\thesection}{A}
\addtocounter{section}{1}
\setcounter{equation}{0}

The goal of this appendix is the following theorem:
\begin{thm}
\label{main}
Let $A$ be a simple abelian variety defined over $K$, and suppose that 
$\EE := \End_K(A) = \End_{\bar K}(A)$.  
There is a positive density set $S$ of rational primes such that for every 
prime $\lambda$ of $\M$ lying above $S$ we have:
\begin{enumerate}
\item
there is a $\tau_0 \in G_{K^\ab}$ such that $A[\lambda]^{\ld\tau_0\rd} = 0$, 
\item
there is a $\tau_1 \in G_{K^\ab}$ such that $A[\lambda]/(\tau_1-1)A[\lambda]$ 
is a simple $\EE/\lambda$-module. 
\end{enumerate}
\end{thm}

The idea of the proof is as follows.  For simplicity, let us assume 
$\End_{\bar K} (A) = \Z$ and further that $K$ is ``large enough''.
Let $\Gamma_\ell$ denote  the image of $\Gal(\bar K/K)$ in $\GL_n(\F_\ell) = \Aut(A[\ell])$,
where $n=2\dim A$.  Using results of Nori, Serre, and  Faltings 
(see Proposition~\ref{serreprop} below), we can show that there exists an 
absolutely irreducible, closed, connected, reductive subgroup $G_\ell\subset\GL_n$ 
such that $\Gamma_\ell$ is a subgroup 
of $G_\ell(\F_\ell)$ of index $\le C$, where $C$ depends only on $n$.  

Using Serre's theory of Frobenius tori, we can find a finite extension $L$ over $K$ such that
if $\ell$ splits completely in $L$, then $G_\ell$ is a split group.
The elements $\tau_0$ and $\tau_1$ which we seek lie in the derived group of $G_K$,
so their images $\bar\tau_0$ and $\bar\tau_1$ in $\Gamma_\ell\subset \Aut(A[\ell])$ 
lie in $[\Gamma_\ell,\Gamma_\ell]$,
i.e., in the group of 
$\F_\ell$-points of the derived group $H_\ell$ of $G_\ell$, which is connected, 
split, and semisimple.   Roughly, we want to show that 
$H_\ell(\F_\ell)\subset \GL_n(\F_\ell)$ has two elements which 
have $0$ and $1$ Jordan blocks of eigenvalue $1$ respectively.
Such elements need not exist in general.
There exist split semisimple groups $H_\ell$ with absolutely irreducible 
representation $V$ such that every element of $H_\ell(\F_\ell)$ has an 
invariant space of dimension $\ge 2$ in $V$.  For instance, $H_\ell$ can be 
a split semisimple group of rank $\ge 2$ and $V$ can be the adjoint representation.

We use a theorem of Pink \cite{Pink} to rule out examples of this kind; from his 
result it is fairly easy to find elements for which $1$ is not an eigenvalue.  
To get a $1$-dimensional $1$-eigenspace is still delicate, however, since $V$ is 
self-dual and of even dimension, so the \emph{multiplicity} of $1$ as an eigenvalue 
is always even.  In particular, a semisimple element cannot have a $1$-dimensional 
$1$-eigenspace.  This makes it necessary to consider elements with non-trivial 
Jordan decomposition.  The construction of such an element is given in Proposition~\ref{group-prop}.

We begin with some estimates useful for guaranteeing the existence of 
sufficiently generic elements in maximal tori over large finite fields 
(i.e., elements whose eigenvalues do not satisfy specified multiplicative conditions). 

\begin{defn}
If $k$ is a positive integer, a subset $S$ of a free abelian group $X$ is \emph{$k$-bounded} 
if there exists a basis $e_i$ of $X$ 
such that each element of $S$ is a linear combination of the $e_i$ with coefficients in $[-k,k]$
\end{defn}

\begin{lem}
\label{minors}
Suppose $X$ is a finitely generated free abelian group, 
and $S$ is a $k$-bounded linearly independent subset of $X$.  
Let $Y$ be the span of $S$, and suppose $Z$ is a subgroup of $X$ 
containing $Y$ with $Z/Y$ finite.  Then
$$
[Z:Y] \le r! k^r
$$
where $r := |S|$.
\end{lem}

\begin{proof}
Without loss of generality we may suppose that $X = \Z^n$ (viewed as row vectors), 
and the basis with respect 
to which the coefficients of $S$ are bounded by $k$ is the standard one.
Let $S = \{s_1,\ldots,s_r\}$, and let $\{z_1,\ldots,z_r\}$ be a basis of $Z$.  
Let $M_Y$ (resp., $M_Z$) be the matrix whose $i$-th row is $s_i$ (resp., $z_i$).  Let 
$N$ be the $r \times r$ matrix representing the $s_i$ in terms of the $z_i$, i.e., 
such that $N M_Z = M_Y$. Then $[Z:Y] = \det(N)$, and $\det(N)$ divides every $r \times r$ minor of $M_Y$.  
Since the entries of $M_Y$ are bounded by $k$, these minors 
are bounded by $r! k^r$.  At least one of them is nonzero, so the lemma follows.
\end{proof}

If $T$ is an algebraic torus then $X^*(T)$ will denote the character group $\Hom(T,\G_m)$.

\begin{lem}
\label{rank-2}
If $T$ is an $r$-dimensional split torus over $\F_\ell$ and $\{\chi_1,\chi_2\}$ is a $k$-bounded 
subset of $X^*(T)$ that generates a rank-$2$ subgroup,
then for all $a_1,a_2\in\F_\ell^\times$, we have
$$
|\{ t\in T(\F_\ell)\mid \chi_1(t) = a_1,\ \chi_2(t)=a_2\}| \le 2k^2 (\ell-1)^{r-2}.
$$
\end{lem}

\begin{proof}
 In the natural bijection between closed subgroups of $T$ and subgroups 
of $X^*(T)$, we have that
$\T:= \ker \chi_1\cap \ker \chi_2\subset T$ corresponds to 
${\mathcal X}:= \langle \chi_1,\chi_2\rangle\subset X^*(T)$, and the identity 
component $\T^\circ$ corresponds to
${\mathcal X}^\circ := ({\mathcal X}\otimes\Q)\cap X^*(T)$.  
As ${\mathcal X}$ has rank $2$, we have $\dim \T = \dim \T^\circ = r-2$, and
$$
[\T:\T^\circ] = |{\mathcal X}^\circ/{\mathcal X}|.
$$
As $\chi_1$ and $\chi_2$ are $k$-bounded, Lemma \ref{minors} shows that 
this index is bounded above 
by $2k^2$, so $\{ t\in T(\F_\ell)\mid \chi_1(t) = a_1,\ \chi_2(t)=a_2\}$ 
(which is either empty or a coset of $\T(\F_\ell)$) satisfies
$$
|\{ t\in T(\F_\ell)\mid \chi_1(t) = a_1,\ \chi_2(t)=a_2\}| 
   \le |\T(\F_\ell)| \le2k^2 (\ell-1)^{r-2}.
$$
\end{proof}

\begin{lem}
\label{zero}
If $G$ is a semisimple group over a field $K$, $(\rho,V)$ is a representation 
of $G$, and there exists $g\in G(K)$ such that $V^{\ld\rho(g)\rd} = (0)$,
then $0$ does not appear as a weight of $\rho$.
\end{lem}

\begin{proof}
Without loss of generality we may assume $K$ is algebraically closed.  
Let $T$ be a maximal torus.  If $0$ appears as a weight of $\rho$,
then $\rho(t)$ has eigenvalue $1$ for all $t\in T(K)$.  The condition 
of having eigenvalue $1$ is conjugation-invariant on $G$, and the union of all 
conjugates of
$T$ includes all regular semisimple elements of $G$ and is therefore 
Zariski-dense.  Thus, $\rho(g)$ has eigenvalue $1$ for all $g\in G(K)$,
and it follows that $V^{\ld\rho(g)\rd}$ is non-trivial.  
\end{proof}

The following proposition gives the key construction of this appendix.
Given a semisimple group $G/\F_\ell$ and an absolutely irreducible $n$-dimensional 
representation $V$ of $G$ defined over $\F_\ell$,
in favorable situations we prove that there exists an element of $G(\F_\ell)$ that fixes a subspace of $V$
of dimension $1$.  If the representation is not self-dual, we can use a semisimple 
element which fixes the highest weight space $W_\eta$
and acts non-trivially on all other weight spaces.  In the self-dual case, we 
find an element whose unique Jordan block with eigenvalue 1
has size $2$, acting on $W_\eta\oplus W_{-\eta}$.

\begin{prop}
\label{group-prop}
For every positive integer $n$, there exists a positive integer $N$ 
such that if $\ell$ is a prime congruent to $1 \pmod{N}$,
$G$ is a simply connected, split semisimple algebraic group over 
$\F_\ell$, and $\rho\colon G \to \GL_n$
is an absolutely irreducible representation such 
that $(\F_\ell^n)^{\ld\rho(g_0\rd)} = 0$ for some $g_0\in G(\F_\ell)$,
then there exists $g_1\in G(\F_\ell)$ such that
$$
\dim (\F_\ell^n)^{\ld\rho(g_1)\rd} = 1.
$$
\end{prop}

\begin{proof}
By replacing $N$ by a suitable multiple, the condition $\ell\equiv 1 \pmod{N}$ 
can be made to imply $\ell$ sufficiently large, so 
henceforth we assume $\ell$ is as large as needed.

We fix a Borel subgroup $B$ of $G$ and a maximal split torus $T$ of $B$, both 
defined over $\F_\ell$.  Every dominant weight $\eta$
of $T$ defines an irreducible representation of $G_{\bar\F_\ell}$, and 
all irreducible representations of $G_{\bar\F_\ell}$ arise in this way.
By a theorem of Steinberg \cite[13.1]{Steinberg}, every irreducible 
$\bar \F_\ell$-representation of $G(\F_\ell)$ is obtained from a unique irreducible
representation $\tilde\rho$ of the algebraic group $G_{\bar\F_\ell}$ 
whose highest weight $\eta = a_1\varpi_1+\cdots+a_r\varpi_r$ can be expressed as a linear
combination of fundamental weights with coefficients $0\le a_i < \ell$.  
By \cite[1.30]{Testerman}, this implies $\max a_i \le n$.  Thus, the set $\Sigma$ of weights of
$\tilde\rho$ (with respect to $T$) is $k$-bounded for some 
constant $k$ depending only on $n$ and the root system of $G$
(and hence, in fact, on $n$ alone).  By Lemma \ref{zero}, $0\not\in \Sigma$, 
so if $|m|>k$ and $\chi\in X^*(T)$, then $m\chi\not\in \Sigma$.  We assume that
$N$ is divisible by $k!$.   We also assume that for all $\chi_1,\chi_2\in \Sigma$ 
distinct, $N$ does not divide $\chi_1-\chi_2$.
This guarantees that for $\chi\in \Sigma$  
$$
\{v\in{ \F}_\ell^n\mid \rho(t)(v) = \chi(t) v\ \forall t\in T({\F}_\ell)\}
$$
is the $\chi$-weight space of the algebraic group $T$.

For each $\chi \in X^*(T)$, we denote by $T_\chi$ the kernel of $\chi$.
Let $d$ be the largest integer such that $\eta\in d X^*(T)$, and let 
$\mu := \eta/d$.  Thus, $\mu$ induces a surjective map $T(\F_\ell)\to \F_\ell^\times$.
As $d \le k$, we have $\ell\equiv 1 \pmod{d}$, so we can fix an element 
$e\in\F_\ell^\times$ of order $d$.  Let $T_{\mu,e}$ denote the translate of $T_\mu$ consisting
of elements $t\in T$ such that $\mu(t)=e$.  The number of $\F_\ell$-points of 
$T_{\mu,e}$ is $(\ell-1)^{r-1}$.  For $\chi\in \Sigma$ not a multiple of 
$\mu$, the intersection $T_{\mu,e}(\F_\ell)\cap T_\chi(\F_\ell)$ has at most 
$2k^2(\ell-1)^{r-2}$ elements by Lemma~\ref{rank-2}.  For $\chi\in \Sigma$
a non-trivial multiple of $\mu$ other than $\pm\eta$, 
$T_{\mu,e}(\F_\ell)\cap T_\chi(\F_\ell)$ is empty.  For $\ell$ sufficiently large, therefore,
$$
T_{\mu,e}(\F_\ell)\setminus \bigcup_{\chi\in \Sigma\setminus \{\pm \eta\}} T_{\chi}(\F_\ell)
$$
has an element $t$.  Thus $\chi(t) \neq 1$ for all $\chi\in \Sigma$ 
except for $\pm \eta$, and $\eta(t)=1$.

If $-\eta\not\in \Sigma$, then setting $g_1=t$, we are done.  We assume, therefore 
that $-\eta\in \Sigma$, so in particular $\rho$ is self-dual.  
If $W_\eta\subset {\F}_\ell^n$ denotes the $\eta$-weight space of $T$ (or equivalently 
$T( {\F}_\ell)$), there exists a unique projection $\pi_\eta\colon  {\F}_\ell^n\to W_\eta$
which respects the $T( {\F}_\ell)$-action and fixes $W_\eta$ pointwise.
Let $U$ be the unipotent radical of $B$.  If there exists $u\in U( {\F}_\ell)$ such that 
$\pi_\eta(\rho(u) w)\neq 0$ for some $w\in W_{-\eta}$
then setting  $g_1 = tu$, we are done.  

We assume henceforth that $N \ge 3(h-1)$ where $h$ denotes the Coxeter number of $G$.  
An upper bound for $h$ is determined by $n$.
By the Jacobson-Morozov theorem in positive characteristic  (cf.\ \cite{SpSt}), 
$\ell>N$ implies that there exists a \emph{principal}
homomorphism $\phi\colon \SL_2\to G$.  Conjugating, we may assume  that the Borel 
subgroup $B_{\SL_2}$ lies in $B$ and the maximal torus $T_{\SL_2}\subset \SL_2$ lies in $T$.  
We identify $X^*(T_{\SL_2})$ with $\Z$ so that positive weights of $T$ restrict to positive weights of
$T_{\SL_2}$.  By definition of principal homomorphism, the restriction of every 
simple root of $G$ with respect to $T$ to $T_{\SL_2}$ equals $2$.
Thus, the restriction $j$ of $\eta$ to $T_{\SL_2}$ is strictly larger than 
the restriction of any other element of $\Sigma$ to $T_{\SL_2}$, and $-j$ is the smallest value 
obtained by restricting elements of $S$ to $T_{\SL_2}$.  The restriction of $V$ to $\SL_2$ is semisimple
when $\ell$ is large by \cite{Larsen} (see also \cite{Jantzen}), and by definition of $j$, $V|_{\SL_2}$ is
a direct sum of one representation $V_1$ of $\SL_2$ of degree $j+1$ and other representations 
of strictly smaller degrees.
The weight spaces $W_\eta$ and $W_{-\eta}$ are contained in $V_1$.
It suffices to find $u$ in $B_{\SL_2}( {\F}_\ell)\cap U( {\F}_\ell)$ and 
$w\in W_{-\eta}\subset V_1$ such that $\pi_\eta(\rho(u) w)\neq 0$.
As
$$
\mathrm{Sym}^{j-1}\begin{pmatrix}1&1\\ 0&1\end{pmatrix}
=\begin{pmatrix}
1&j-1&\binom{j-1}2&\cdots&1 \\
0&1&j-2&\cdots&1 \\
0&0&1&\cdots&1\\
\vdots&\vdots&\vdots&\ddots&\vdots \\
0&0&0&\cdots&1
\end{pmatrix},
$$
any non-trivial $u$ and $w$ will do.
\end{proof}

\begin{lem}
\label{Small-index}
Fix a positive integer $B$.   
Suppose $H$ is a connected reductive algebraic group over $\F_\ell$, and 
$\Gamma$ is a subgroup of $H(\F_\ell)$ of index $\le B$.
Let $\tilde H$ denote the universal covering group of the derived group of $H$, and 
$\pi_\ell : \tilde H(\Fl) \to H(\Fl)$ the covering map.
If $\ell$ is sufficiently large in terms of $B$, then 
the derived group of $\Gamma$ contains the image of $\pi_\ell$.  
\end{lem}

\begin{proof}
Let $\tilde\Gamma = \pi_\ell^{-1}(\Gamma) \subset \tilde H(\Fl)$, 
so $[\tilde H(\Fl):\tilde\Gamma] \le B$.
If $\ell$ is sufficiently large, then the quotient of $\tilde H(\F_\ell)$ by its center is a product $\Pi$
of finite simple groups (\cite[Theorems~5 and 34]{LCG}), and $\tilde H(\F_\ell)$ is a universal central 
extension of this quotient (\cite[Theorems~10 and 34]{LCG}).
Moreover, each factor of $\Pi$ is a quotient group of $\tilde H(\Fl)$, 
is therefore generated by elements of $\ell$-power order
(\cite[Theorem~12.4]{Steinberg}), and therefore has order at 
least $\ell$.  If $\tilde \Gamma$ 
is a proper subgroup of $\tilde H(\Fl)$, then its image in $\Pi$ is a proper subgroup of 
index $\le B$, which is impossible if $\ell > B!$.  Thus if $\ell$ is sufficiently large, 
we conclude that $\tilde\Gamma = \tilde H(\Fl)$, and so $\pi_\ell(\tilde H(\Fl)) \subset \Gamma$ and 
(since $\tilde H(\Fl)$ is perfect),
$$
\pi_\ell(\tilde H(\Fl)) = [\pi_\ell(\tilde H(\Fl)),\pi_\ell(\tilde H(\Fl))] 
   \subset [\Gamma,\Gamma].
$$
\end{proof}

Fix a simple abelian variety $A$ defined over a number field $K$.  
Let $\EE := \End_K(A)$, let $\OO$ denote the center of $\EE$, and $\M = \OO \otimes \Q$.  
Since $A$ is simple, $\M$ is a number field and $\OO$ is an order in $\M$.
Suppose $\ell$ is a rational prime not dividing the discriminant of $\OO$, 
such that $\ell$ splits completely in $\M/\Q$, 
and $\lambda$ is a prime of $\M$ above $\ell$.  We will abbreviate 
$$
 \EElambda := \EE \otimes_\OO \M_\lambda, \quad 
   \EE/\lambda := \EE \otimes_\OO \OO/\lambda.
$$

We assume from now on that $K$ is large enough so that 
$$
\EE := \End_K(A) = \End_{\bar K}(A)
$$
and $\ell$ is large enough (Proposition 9.1) so that  
$$
\EElambda \cong M_d(\Ql) \quad \text{and} \quad \EE/\lambda \cong M_d(\Fl)
$$
where for a field $F$, $M_d(F)$ denote the simple $F$-algebra of $d \times d$ 
matrices with entries in $F$.
Let $V_\lambda(A)$ denote the $\lambda$-adic Tate module
$$
V_\lambda(A) := (\displaystyle\lim_\leftarrow A[\lambda^k]) \otimes_{\OO_\lambda} \M_\lambda,
$$ 
let $W_\lambda$ (resp., $\bar W_\lambda$) denote the unique (up to isomorphism) simple 
$\EElambda$-module (resp., $\EE/\lambda$-module), 
and define
$$
X_\lambda = \Hom_{\EElambda}(W_\lambda,V_\lambda(A)), \quad 
   \bar X_\lambda = \Hom_{\EE/\lambda}(\bar W_\lambda,A[\lambda]).
$$
Then $X_\lambda$ is a $\Ql$-vector space of dimension $n$, and $\bar X_\lambda$ 
is an $\Fl$-vector space of dimension $n$, where  
$$
n := \lth_{\EElambda} V_\lambda(A) = \lth_{\EE/\lambda}A[\lambda] = \frac{2 \dim(A)}{d}.
$$
There is a natural Galois action on $X_\lambda$ and $\bar X_\lambda$, 
where we let $G_K$ act trivially on $W_\lambda$ and $\bar W_\lambda$.  Denote by 
$$
\rho_\lambda : G_K \to \Aut(X_\lambda) \cong \GL_n(\Ql), \quad 
   \bar\rho_\lambda : G_K \to \Aut(\bar X_\lambda) \cong \GL_n(\Fl),
$$ 
the corresponding representations.

\begin{lem}
\label{samething}
There are natural $G_K$-equivariant isomorphisms
$$
\End_{\Ql}(X_\lambda) \cong \End_{\EElambda}(V_\lambda(A)), \quad 
   \End_{\Fl}(\bar X_\lambda) \cong \End_{\EE/\lambda}(A[\lambda]).
$$
\end{lem}

\begin{proof}
The map $\End_{\EElambda}(V_\lambda(A)) \times X_\lambda \to X_\lambda$ 
given by $(f,\varphi) \mapsto f \circ \varphi$ induces an injective 
homomorphism $\End_{\EElambda}(V_\lambda(A)) \to \End_{\Ql}(X_\lambda)$.  
Since both spaces have $\Ql$-dimension $n^2$, this map is an isomorphism.
The proof of the second isomorphism is the same.
\end{proof}

Let $G_\lambda \subset \Aut(X_\lambda)$ be the Zariski closure of the image 
$\rho_\lambda(G_K)$.  

\begin{prop}
\label{serreprop}
Replacing $K$ by a finite extension if necessary, for all $\ell$ sufficiently large we have:
\begin{enumerate}
\item
$G_\lambda$ is a connected, reductive, absolutely irreducible subgroup of $\Aut(X_\lambda)$, 
with center equal to the group of scalars $\G_m$,
\item
there is a connected, reductive, absolutely irreducible subgroup $H_\lambda$ of 
$\Aut(\bar X_\lambda)$, with center equal to the group of scalars $\G_m$, such that 
\begin{enumerate}
\item
the image $\bar\rho_\lambda(G_K)$ is contained in 
$H_\lambda(\Fl)$ with index bounded independently of $\lambda$ and $\ell$,
\item
the rank of $H_\lambda$ is equal to the rank of $G_\lambda$ (and is independent of $\lambda$ and $\ell$).
\end{enumerate}
\end{enumerate}
\end{prop}

\begin{proof}
Using Lemma \ref{samething}, we can identify $G_\lambda$ with the Zariski closure of 
the image of $G_K$ in $\Aut_{\EElambda}(V_\lambda(A)) \subset \Aut_{\Ql}(V_\lambda(A))$.
The fact that $G_\lambda$ is reductive and connected (after possibly increasing $K$) 
now follows from a combination of Faltings' theorem and a theorem of 
Serre \cite[\S 2.2]{Course84-85}.

It also follows from Faltings' theorem that the commutant of $G_\lambda$ in  
$\Aut_{\Ql}(V_\lambda(A))$ is $\EElambda$, and hence the commutant of 
$G_\lambda$ in $\End_{\EElambda}(V_\lambda(A)) = \End(X_\lambda)$ 
is the center of $\EElambda$, which is $\Ql$.  This shows that $G_\lambda$ 
is absolutely irreducible, and since $G_\lambda$ contains the scalar matrices 
\cite{bogomolov} this completes the proof of (i).

The proof of (ii) is similar.  The definition of the connected reductive group
$H_\lambda \subset \Aut_{\EE/\lambda}(A[\lambda]) \subset \Aut_{\Fl}(A[\lambda])$ 
is given by Serre in \cite[\S 3]{Vigneras}. 
The fact that $H_\lambda$ is absolutely irreducible, and that the center of 
$H_\lambda$ is $\G_m$, follows as for (i): 
Remark 4 at the end of \cite[\S 3]{Vigneras} shows that the commutant of 
$H_\lambda$ in $\Aut_{\Fl}(A[\lambda])$ is $\EE/\lambda$, so the 
commutant of $H_\lambda$ in $\Aut_{\EE/\lambda}(A[\lambda]) = \Aut(\bar X_\lambda)$ 
is the center of $\EE/\lambda$, which is $\Fl$.  That $H_\lambda$ containes the homotheties 
is \cite[\S5]{Vigneras}.

Th\'eor\`emes 1 and 2 of \cite{Vigneras} give (a) and (b) of (ii).
\end{proof}

From now on suppose that $K$ and $\ell$ are large enough to satisfy Proposition \ref{serreprop}, 
and let $H_\lambda$ be as in Proposition \ref{serreprop}(ii).
Let $\tilde G_\lambda$ and $\tilde H_\lambda$ denote the simply connected cover of 
the derived group of $G_\lambda$ and $H_\lambda$, respectively.  

\begin{lem}
\label{r-lemma}
There is a positive integer $r$, independent of $\lambda$ and $\ell$, such that 
for every $h \in H_\lambda(\Fl)$, we have 
$h^{nr}/\det(h)^r \in \image(\tilde{H}_\lambda(\Fl) \to H_\lambda(\Fl))$.
\end{lem}

\begin{proof}
By Proposition \ref{serreprop}(i), we have $H_\lambda = \G_m \cdot \SH_\lambda$ where 
$\SH_\lambda$, the derived group of $H_\lambda$, is $H_\lambda \cap \SL_n(\tilde X_\lambda)$. 
We have $h^{nr}/\det(h)^r \in \SH_\lambda(\Fl)$ for every $r$, so to prove the lemma we need 
only show that the cokernel of $\pi : \tilde H_\lambda(\Fl) \to \SH_\lambda(\Fl)$ is 
bounded by a constant depending only on $n$.

It follows from Lang's theorem (cf.\ \cite[Proposition 16.8]{borel}) that the kernel 
and cokernel of $\pi$ have the same order.  The kernel of $\pi$ is a subgroup of the center of 
$\tilde H_\lambda$, and the order of the center of a semisimple group can be bounded only in terms 
of its root datum.  (Indeed, this can be checked over an algebraically closed field; the center
lies in the centralizer $T$ of every maximal torus $T$ and in the point stabilizer $\ker\alpha\subset T$
of every root space $U_\alpha$ of $T$.)

\end{proof}

\begin{lem}
\label{nxt-lemma}
The representation of $\tilde G_\lambda$ on $X_\lambda$ does not have $0$ as a weight.
\end{lem}

\begin{proof}
By \cite[Corollary 5.11]{Pink}, the highest weight of $G_\lambda$ acting on
$X_\lambda$ is minuscule; i.e., the weights form an orbit under the Weyl group.  
Any weight which is trivial on the derived group of $G_\lambda$ is fixed by the Weyl group of $G_\lambda$;
as the representation $X_\lambda$ is faithful, no such weight can occur.  Regarding $X_\lambda$
as a representation of $\tilde G_\lambda$, it factors through $G_\lambda$, so
again, there can be no zero weight.
\end{proof}

\begin{prop}
\label{rlv-prop}
Suppose $r$ is a positive integer.
If $\ell$ is sufficiently large then there is a prime $v \nmid \ell$ of $K$ 
such that (writing $\Frob_{v}$ for a Frobenius automorphism at $v$)
\begin{enumerate}
\item
$A$ has good reduction at $v$ and at all primes above $\ell$,
\item
$\rho_\lambda(\Frob_{v}) \in G_\lambda(\Ql)$ generates a Zariski 
dense subgroup of the unique maximal torus to which it belongs,
\item
$\det(\rho_\lambda(\Frob_{v}^{nr})/\det(\rho_\lambda(\Frob_{v}^r))-1) \ne 0$.
\end{enumerate}
\end{prop}

\begin{proof}
By Proposition \ref{serreprop}(i), $G_\lambda$ contains all scalar matrices.  It
follows from Lemma \ref{nxt-lemma} (as in the proof of Lemma \ref{r-lemma}) 
that the condition that $\det(g)^r$ 
is an eigenvalue of $g^{rn}$ does not hold on all of $\tilde G_\lambda$, 
so it does not hold on all of $G_\lambda$, so it holds on a proper closed subset of $G_\lambda$.

By \cite{Ribet-1-1-81}, there is a dense open subset $U$ of $G_{\lambda}$ such that 
$\rho_\lambda(\Frob_{v}) \in U(\Q_{\ell})$ implies that $\rho_\lambda(\Frob_{v})$ 
generates a Zariski-dense subgroup of the unique maximal torus 
to which it belongs.  By Chebotarev density, there exists $v$ such that $g := \rho_\lambda(\Frob_{v})$ 
satisfies this condition together with the condition
that $\det(g)^r$ is not an eigenvalue of $g^{rn}$.
\end{proof}

Fix $\lambda_0 \mid \ell_0$ and $v$ satisfying Proposition \ref{rlv-prop}, and define 
$$
g_0 := \rho_{\lambda_0}(\Frob_{v}) \in \Aut(X_{\lambda_0}).
$$
Let $P_v(x) \in \Z[x]$ be the characteristic polynomial of $g_0$, 
which is independent of the choice of $\ell_0$ and $\lambda_0$, and let 
$L$ denote the splitting field of $P_v(x)$ over $\Q$.  
Let $\Sigma$ denote the set of distinct weights of $G_{\lambda_0}$ with respect to the 
(unique, maximal) torus containing $g_0$.

Fix $r$ as in Proposition \ref{r-lemma}.  Without loss of generality we may assume that 
$r$ is divisible by $(n-1)!$.  Let $\gamma_0 := g_0^{nr}/\det(g_0^r)$, and define 
$$
\mu := \prod_{\chi\in\Sigma}(\chi(\gamma_0)-1)
   \prod_{\chi,\chi'\in\Sigma, \chi\ne\chi'}(\chi(\gamma_0)-\chi'(\gamma_0))
$$

\begin{lem}
\label{A10}
We have $\mu \ne 0$.
\end{lem}

\begin{proof}
By Proposition \ref{rlv-prop}(iii), $1$ is not an eigenvalue of $\gamma_0$, 
so $\chi(\gamma_0) \ne 1$ for every weight $\chi$.  Since $\rho_{\lambda_0}(\Frob_{v})$ 
generates a Zariski dense subgroup of the maximal torus that contains it, so does 
$\Frob_{v}^{nr}$.  Hence if $\chi \ne \chi' \in \Sigma$, then 
$\chi(\rho_{\lambda_0}(\Frob_{v}^{nr})) \ne \chi'(\rho_{\lambda_0}(\Frob_{v}^{nr}))$ 
and $\chi(\gamma_0) \ne \chi'(\gamma_0)$.
\end{proof}

\begin{prop}
\label{A11}
Suppose $\ell$ splits completely in $L/\Q$ and $\ell$ does not divide 
$
\bN_{L/\Q}\mu.
$
Then $\tilde H_\lambda$ is split, and there is an $\eta_0$ in the image of the map 
$\tilde H_\lambda(\Fl) \to H_\lambda(\Fl)$ such that $(\bar X_\lambda)^{\ld\eta_0\rd} = 0$.
\end{prop}

\begin{proof}
Let $h_0 = \bar\rho_\lambda(\Frob_{v}) \in H_\lambda(\Fl)$, and let $\bar{P}_v(x) \in \Fl[x]$ be the characteristic 
polynomial of $h_0$.  
Then $\bar{P}_v(x)$ is the reduction of $P_v(x)$ modulo $\lambda$.

Let $h_0 = su$ be the Jordan decomposition of $h_0$, with $s$ semisimple and $u$ unipotent, 
and $Z$ a maximal torus of $H_\lambda$ such that $s \in Z(\F_\ell)$.  
Since $\ell$ splits completely in $L/K$, all roots of $\bar P(x)$ lie in $\Fl$, 
and distinct weights correspond to distinct eigenvalues.
If $\bar\chi$, $\bar\chi'$ are weights of $H_\lambda$ with respect to $Z$, and 
$\Frob_\lambda(\bar\chi) = \bar\chi' \neq \bar\chi$, then $\bar\chi(s)\in \F_\ell$ implies that 
$\bar\chi(s)= \bar\chi'(s)$, contrary to assumption.
Thus $\Frob_\lambda$ acts trivially on the weights of $H_\lambda$.  
It follows that $\Frob_\lambda$ acts trivially on $Z$, which means $H_\lambda$ is split, and 
therefore $\tilde H_\lambda$ is split.

Let $\eta_0 = h_0^{nr}/\det(h_0^r)$, so $\eta_0$ is in the image of  
$\tilde H_\lambda(\Fl) \to H_\lambda(\Fl)$ by Lemma \ref{r-lemma}(ii).  The eigenvalues of $\gamma_0$ 
are the values $\chi(\gamma_0)$ for $\chi \in \Sigma$, and the eigenvalues of $\eta_0$ 
are the reductions of those values modulo $\lambda$.  By assumption none of those 
values reduce to $1$, so $1$ is not an eigenvalue of $\eta_0$ and $(\bar X_\lambda)^{\ld\eta_0\rd} = 0$.
\end{proof}

\begin{prop}
\label{irredlem}
The representation $\pi_\ell : \tilde H(\Fl) \to H(\Fl) \subset \GL_n(\Fl)$ 
is absolutely irreducible.
\end{prop}

\begin{proof}
By Proposition \ref{serreprop}(ii), the subgroup $H_\lambda(\F_\ell) \subset \GL_n(\Fl)$ 
is absolutely irreducible.
By functoriality, the image $\pi_\ell(\tilde H_\lambda(\Fl))$ 
is a normal subgroup of $H_\lambda(\F_\ell)$.  
If $\pi_\ell$ is not absolutely irreducible, then there is a decomposition
$$
\bar\F_\ell^n = \dirsum{}{Z_i}
$$ 
where each $Z_i$ is an irreducible $\pi_\ell(\tilde H_\lambda(\F_\ell))$-module 
and the $Z_i$ are permuted transitively by the action of 
$H_\lambda(\F_\ell)/\pi_\ell(\tilde H_\lambda(\F_\ell))$.
The number of irreducible summands is bounded by the dimension $n$, 
so for every $g \in H_\lambda(\F_\ell)$, every eigenvalue of $g^{n!}$
occurs with multiplicity greater than $1$.

Since $g_0$ generates a Zariski dense subgroup of the unique maximal torus in $G_{\lambda_0}$ 
that contains it, so does $g_0^{n!}$.  It follows that the eigenvalue of $g_0^{n!}$ 
corresponding to the highest weight has multiplicity $1$.  Since $\ell \nmid \mu$, 
the eigenvalues of $g_0^{n!}$ are distinct modulo $\lambda$, so one of the 
eigenvalues of $\bar\rho_\lambda(\Frob_v^{n!})$ has multiplicity $1$.  
This contradiction shows that $\pi_\ell$ is absolutely irreducible.
\end{proof}

\begin{cor}
\label{irredcor}
If $\ell$ is sufficiently large then $A[\lambda]$ is an irreducible $\EE[G_K]$-module.
\end{cor}

\begin{proof}
By Lemma \ref{Small-index} applied with $H := H_\lambda$ and $\Gamma := \bar\rho_\lambda(G_K)$, 
and Proposition \ref{serreprop}(ii)(a), the image of $G_K$ in $H_\lambda(\Fl)$ 
contains the image of $\tilde H_\lambda(\Fl)$.  By Proposition \ref{irredlem} and Lemma \ref{samething} 
the latter is an irreducible subgroup of $\Aut(\bar X_\lambda) = \Aut_{\EE/\lambda}(A[\lambda])$.
\end{proof}

We can now prove Theorem~\ref{main}.

\begin{proof}
Since $\End_K(A) = \End_{\bar K}(A)$, we have that $A$ is absolutely 
simple and increasing $K$ does not change $\EE$.
Thus it suffices to prove the theorem with $K$ replaced by a finite extension, if necessary, 
so we may assume that $K$ and $\ell$ satisfy Proposition \ref{serreprop}.

Suppose now that $\ell$ splits completely in $\M$ 
and in the number field $L$ defined before Lemma \ref{A10}, and that $\ell \equiv 1 \pmod{N}$ 
where $N$ is as in Proposition \ref{group-prop}.
We will apply Proposition \ref{group-prop} with $G = \tilde H_\lambda$, and 
the representation $\rho = \pi_\ell : \tilde H_\lambda \to H_\lambda \subset \GL_n$.
By Proposition \ref{A11}, $\tilde H_\lambda$ is split 
and there is an $\eta_0 \in \tilde H_\lambda(\Fl)$ such that 
$(\bar X_\lambda)^{\ld\pi_\ell(\eta_0)\rd} = 0$.  
By Proposition \ref{irredlem}, $\pi_\ell$ is absolutely irreducible.
Thus we can apply Proposition \ref{group-prop} to conclude that there is an 
$\eta_1 \in \tilde H_\lambda(\Fl)$ such that $\dim_{\Fl}(\bar X_\lambda)^{\ld\pi_\ell(\eta_1)\rd} = 1$.

By Lemma \ref{Small-index} (applied with $H := H_\lambda$ and $\Gamma := \bar\rho_\lambda(G_K)$) 
and Proposition \ref{serreprop}(ii)(a), 
for all sufficiently large $\ell$ we have  
$\pi_\ell(\tilde H_\lambda) \subset \bar\rho_\lambda(G_{K^\ab})$.  In particular we can choose 
$\tau_i \in G_{K^\ab}$ so that $\bar\rho_\lambda(\tau_i) = \pi_\ell(\eta_i)$ for $i = 0, 1$.
We have
$$
(\bar X_\lambda)^{\ld\tau_i\rd} 
   = \Hom_{\EE/\lambda}(\bar W_\lambda,A[\lambda])^{\ld\tau_i\rd}
   = \Hom_{\EE/\lambda}(\bar W_\lambda,A[\lambda]^{\ld\tau_i\rd})
$$
so 
$$
\lth(A[\lambda]^{\ld\tau_i\rd}) 
   = \dim(\bar X_\lambda)^{\ld\tau_i\rd} = i
$$
for $i = 0, 1$.  This proves the theorem.
\end{proof}

\bigskip\small\noindent
{\sc Department of Mathematics, Indiana University, Bloomington, IN 47405, USA}
{\em E-mail address:} {\href{mailto:larsen@math.indiana.edu}{\tt larsen@math.indiana.edu}}

\end{document}